\documentclass[aihp]{imsart}

\RequirePackage{amsthm,amsmath,amsfonts,amssymb}
\RequirePackage[numbers]{natbib}

\usepackage{amsmath,amssymb,amsthm,amsfonts,mathrsfs,amsxtra,bbm,comment}

\usepackage{tikz-cd}

\usepackage[hyperfootnotes=true,
        pdffitwindow=true,
        plainpages=false,
        pdfpagelabels=true,
        pdfpagemode=UseOutlines,
        pdfpagelayout=SinglePage,
        hyperindex,]{hyperref}

\numberwithin{equation}{section}

\startlocaldefs
\theoremstyle{plain}
	\newtheorem{theorem}{Theorem}
	\newtheorem{lemma}[theorem]{Lemma}
	\newtheorem{corollary}[theorem]{Corollary}
	\newtheorem{proposition}[theorem]{Proposition}
\theoremstyle{definition}
	\newtheorem{definition}[theorem]{Definition}
	
\theoremstyle{remark}

\newcommand{\Z}{\mathbb Z}
\newcommand{\R}{\mathbb R}
\newcommand{\C}{\mathbb C}

\newcommand{\U}{\mathrm U}

\newcommand{\Herm}{\mathrm{Herm}}

\newcommand{\E}{\mathbb{E}}
\newcommand{\dv}{\mathrm{d}}
\newcommand{\tr}{\mathrm{Tr}\,}
\newcommand{\eig}{\mathrm{eig}}
\newcommand{\diag}{\mathrm{diag}}

\newcommand{\s}{\mathcal S}
\newcommand{\f}{\mathcal F}

\newcommand{\m}{\mathcal M}
\newcommand{\sgn}{\mathrm{sgn}}
\newcommand{\la}{\left\langle}
\newcommand{\ra}{\right\rangle}

\endlocaldefs

\begin{document}

\begin{frontmatter}

\title{Stable distributions and domains of attraction for unitarily invariant Hermitian random matrix ensembles}

\runtitle{Stable distributions and domains of attraction for unitarily invariant Hermitian random matrix ensembles}

\begin{aug}
\author[A]{\inits{MK}\fnms{Mario} \snm{Kieburg}\ead[label=e1]{m.kieburg@unimelb.edu.au}}
\and
\author[A,B]{\inits{JZ}\fnms{Jiyuan} \snm{Zhang}\ead[label=e2]{jiyuanzhang.ms@gmail.com}}
%%%%%%%%%%%%%%%%%%%%%%%%%%%%%%%%%%%%%%%%%%%%%%
%% Addresses                                %%
%%%%%%%%%%%%%%%%%%%%%%%%%%%%%%%%%%%%%%%%%%%%%%
\address[A]{School of Mathematics and Statistics, University of Melbourne, 813 Swanston Street, Parkville, Melbourne VIC 3010, Australia. \printead{e1,e2}}
\address[B]{Department of Mathematics, KU Leuven, Celestijnenlaan 200B, B-3001 Leuven, Belgium.}
\end{aug}

\begin{abstract}
We consider random matrix ensembles on the set of Hermitian matrices that are heavy tailed, in particular not all moments exist, and that are invariant under the conjugate action of the unitary group. The latter property entails that the eigenvectors are Haar distributed and, therefore, factorise from the eigenvalue statistics. We prove a classification for stable matrix ensembles of this kind of matrices represented in terms of matrices, their eigenvalues and their diagonal entries with the help of the classification of the multivariate stable distributions and the harmonic analysis on symmetric matrix spaces. Moreover, we identify sufficient and necessary conditions for their domains of attraction. To illustrate our findings we discuss for instance elliptical invariant random matrix ensembles and P\'olya ensembles, the latter playing a particular role in matrix convolutions. As a byproduct we generalise the derivative principle on the Hermitian matrices to general tempered distributions. This principle relates the joint probability density of the eigenvalues and the diagonal entries of the random matrix.
\end{abstract}

\begin{abstract}[language=french]
\
\end{abstract}


\end{frontmatter}

\maketitle

\section{Introduction}

\subsection{Introductory Remarks}

Heavy-tailed random matrices have been of interest, lately, since they have applications in machine learning~\cite{MM2019,MM2021}, disordered systems~\cite{BCP2008,BT2012}, quantum field theory~\cite{Kanazawa2016}, finance~\cite{BJNPZ2001,BJNPZ2004,MS2003,AFV2010,MSG2014} and statistics in general~\cite{BGW2006,HM2019,Heiny2019,HM2017,Oymak2017,Minsker2018}. Especially, Wigner matrices with heavy-tailed and independently distributed matrix entries have been studied in detail~\cite{CB1994,Soshnikov2004,BBP2007,BJ2011,BGM2014,BJNPZ2007,BG2008,ABP2009,Vershynin2012,BM2016,TBT2016,AM2001,BG2017,Male2017,PSFG2010}. Let us recall that a heavy-tailed random variable is one where not all generalised moments (exponents do not need to be positive integers). This has drastic consequences for the spectral statistics of random matrices. For instance, it was shown~\cite{Soshnikov2004,BBP2007,ABP2009,TBT2016,BT2012} that heavy-tailed Wigner matrices show a transition to Poisson statistics of the eigenvalues in the tail of the level density when the matrix dimension tends to infinity. The corresponding eigenvectors become localised for these eigenvalues~\cite{TBT2016,BG2017,BT2012}. Other heavy-tailed ensembles where used for band matrices~\cite{BP2014}, correlated matrices~\cite{BBP2007,WWKK2016,GS2020} as well as group invariant random matrices~\cite{WF2000,Tierz2001,BJJNGZ2002,AV2008,AAV2009,CM2009,CM2010}.

Very recently, it was shown~\cite{KM2021} that also unitarily invariant Hermitian random matrices can exhibit Poisson statistics or a mixture of Poisson and sine-kernel statistics in the tail of the level density. This was corroborated by Monte-Carlo simulations and supersymmetry computations. Thus, the Poisson statistics is not restricted to the localisation of the eigenvectors. Even non-trivial mesoscopic spectral statistics, meaning on a scale between the mean level spacing of consecutive eigenvalues and the scale of the macroscopic level density, like the Wigner semicircle~\cite{semicircle} or the Mar\v{c}enko-Pastur distribution~\cite{MPdist}, have been observed in~\cite{KM2021}. Born out from these findings, new question have arisen. What is the true source for the transition between Poisson statistics in the tail and sine-kernel in the bulk of the spectrum? What is the source of the mesoscopic spectral statistics? Are those statistics universal, meaning which class of random matrices exhibit those? And most importantly, when we add infinitely many independent copies of the same random matrix do we obtain central limit theorems? What is the classification of the corresponding stable statistics and what is their domains of attraction?

Certainly, when considering the limit of large matrix dimensions complicates the situation as double scaling limits may exist. It is especially problematic that even no classification of generalised central limit theorems at finite matrix dimension exists, namely central limit theorems of the following form: let $X_i\in\Herm(N)$ be independent copies of the Hermitian  random matrix $X\in\Herm(N)$ that is unitarily invariant, i.e., $UXU^\dagger$ and $X$ being equal in distribution. Here, we employ the notation of $\Herm(N)$ which is the set of $N\times N$ Hermitian matrices, $U^\dagger$ is the Hermitian adjoint of the unitary matrix $U\in\U(N)$ and $\U(N)$ is the set of $N\times N$ unitary matrices. Then, there exist two sequences $\{A_m\}_{m\in\mathbb{N}}\subset\Herm(N)$ and $\{B_m\}_{m\in\mathbb{N}}\in\R_+=(0,\infty)$ such that
\begin{equation}\label{def.CLT}
	\lim_{m\to\infty}\left[\frac{X_1+\ldots+X_m}{B_m}-A_m\right]=Y
\end{equation}
with $Y$ a random matrix that is stable. Random matrices satisfying this limit are said to be in the domain of attraction of $Y$. We call matrix $Y$ to be stable if the distribution of the sum $a Y_1+b Y_2$ of two statistically independent copies $Y_1$ and $Y_2$ of  $Y$ and any two constants $a,b>0$ agrees with $Y$ again after properly rescaling and shifting, see~\cite{Shimura}. In this way, we guarantee that all kinds of statistics, say of eigenvalues, eigenvectors, matrix entries, sub-blocks or traces and determinants, will not change when adding independent copies of it.

We aim at filling this gap of understanding such central limit theorems in the present work. Particularly, we consider  unitarily invariant Hermitian random matrices with heavy tails. Well-known heavy tailed unitarily invariant random matrices are the Cauchy ensemble~\cite{WF2000,Tierz2001}, the Cauchy-Lorentz ensemble~\cite{WWKK2016} (also known as the matrix Student t-function in statistics~\cite{JA96,BBP2007b,Sutradhar86}), and the Hermitised products of Ginibre with inverse Ginibre matrices~\cite{ARRS2016,Forrester2014,LWZ2016}. The latter are only one kind of several multiplicative P\'olya ensembles~\cite{KK2016,KK2019,FKK2021} that can exhibit heavy tails. 

When it comes to a sum of independent and identically distributed random variables, the most prominent central limit theorem states that the sum converges to a normal distribution after a proper shift and scaling. A suitable condition is when those random variables have finite second moments. A generalisation for this classical central limit theorem is the limit to stable distributions like the L\'evy $\alpha$-stable distributions in the univariate case~\cite{GKC}. In the multivariate case such limits were investigated and classified in~\cite{ST94,Rvaceva,Shimura}. Their main tool has been harmonic analysis, especially the Fourier transform on $\R^d$. In probability theory this transform is known as the characteristic function.

We will pursue these ideas and extend them using the knowledge of harmonic analysis and the Fourier transform on symmetric matrix spaces which is also know as the spherical transform~\cite{Helgason2000}. The reason why we go over to these tools, despite the fact that $\Herm(N)$ is isomorphic to $\R^{N^2}$, is that the eigenvector statistics are, evidently, always given by the Haar measure on the coset $\U(N)/\U^N(1)$ (the division of $\U(N)$ by $\U^N(1)=\U(1)\times\ldots\times\U(1)$ corresponds to the diagonal matrix of phases that commute with diagonal matrices). Hence, there is not much information comprised and we can disregard them. The information about heavy tails and generalised central limit theorems are completely encoded in the eigenvalue statistics. Thus, it is only natural to reduce the focus to these quantities which, however, complicates the computations as the eigenvalues of matrices do not form a vector space as the additive action is done for the whole matrix and not for their eigenvalues, only.

To circumvent this disadvantage, we exploit a one-to-one relation between the joint probability distribution of the eigenvalues and the one of the diagonal entries of a unitarily invariant Hermitian random matrix. This relation is known as a derivative principle~\cite{CDKW14,Faraut,MZB16,ZK2020}. The idea is that the diagonal entries again build a vector space with respect to matrix addition so that the well-known additive convolution formula for vector spaces and the corresponding Fourier transform can be applied. Therefore, the derivative principle transfers a sum of random matrices to a sum of random vectors from which one can deduce the eigenvalue statistics.

\subsection{Main Results and Structure of the Article}

In preparation for proving the following theorems,  we introduce the  tools and our notations in Sec.~\ref{s2}. In particular, we define explicitly what unitarily invariant random matrices are and briefly recall their consequences regarding the induced marginal distributions for the eigenvalues and the diagonal entries in subsection~\ref{s2.1}. In subsection~\ref{s2.2}, we quickly recall how the Fourier transform on $\Herm(N)$ is related to the spherical transform for the marginal distributions of the measures for the eigenvalues. We state these definitions and relations in a form that are also applicable to tempered distributions on the respective spaces since we need those for the proofs of the main theorems. Thus, we do not assume that a probability density always exists. Certainly, Dirac delta distribution are allowed in most of our theorems and propositions, too. This is the reason why we need to generalise the derivative principle to general  Borel probability measures which remain unchanged under any unitary conjugation, i.e., $F(\dv X)=F(U\dv XU^\dagger)$ for all $U\in\U(N)$ and $X\in\Herm(N)$. We define

\newpage

\begin{theorem}[Derivative Principle]\label{p2.2}\

	Let $S_N$ be the symmetric group and $X\in\Herm(N)$ be $\U(N)$-invariant random matrix with a Borel probability distribution $F$. Its corresponding marginal measures for the unordered eigenvalues is $f_\eig$ and for the unordered diagonal entries it is $f_\diag$. Then, there exists a unique Borel measure $w$ on $\R^N$ such that
	\begin{equation}\label{2.3.5}
	\int_{\R^N}\phi(x)f_\eig(\dv x)=\int_{\R^N}\frac{1}{\prod_{j=0}^{N}j!}\Delta(\partial_x)\Big(\Delta(x)\phi(x)\Big)\,w(\dv x)
	\end{equation}
	for any Schwartz function $\phi\in \mathscr S(\R^N)$. The polynomial $\Delta(x)=\prod_{1\leq a<b\leq N}(x_b-x_a)$ is the Vandermonde determinant. Moreover, it holds
	\begin{equation}
		\s f_\eig(s)=\f w(s),
	\end{equation}
	where the spherical transform $\s$ and the Fourier transform $\f$ are defined in subsection~\ref{s2.2}. The measure  $w$ is equal to $f_\diag$ in distribution, that is
	\begin{equation}\label{2.3.6}
	\int_{\R^N} \psi(x)w(\dv x)=\int_{\R^N} \psi(x)f_\diag(\dv x)
	\end{equation}
	for all Schwartz functions $\psi\in \mathscr S_{S_N}(\R^N)$  that are permutation invariant in its $N$ arguments (which is highlighted by the subscript $S_N$).
\end{theorem}

Despite the fact that this theorem is still a preparation for our original goal, which is the classification of stable unitarily invariant random matrix ensembles and the identification of their domains of attraction, we consider it, nonetheless, as our first main result. It is proven in subsection~\ref{sec:der}. So far this theorem was only shown~\cite{CDKW14,Faraut,MZB16,ZK2020} when a  density for the measure exists; it is restated in Proposition~\ref{p2.3.1}.

In Sec.~\ref{s3}, we go over to our main focus namely stable invariant Hermitian random matrices. In subsection~\ref{s3.1}, we introduce the notion of a stable invariant ensembles and recollect what is known for stable random vectors in multivariate statistics. Unitarily invariant stable ensembles are characterised by four ingredients: a shift $y_0I_N$ that can only come as a real multiple ($y_0\in\mathbb{R}$) of the identity matrix $I_N$ due to the invariance, a scaling $\gamma>0$ that describes how broad the function is, the stability exponent $\alpha\in(0,2)$ that states how heavy-tailed the distribution is, and  a normalised spectral measure $H$, see~\cite{Rvaceva,ST94,Shimura}, which appears in the logarithm of their characteristic functions and describes the anisotropy of the distribution. We denote this quadruple of a stable ensemble by $S(\alpha,\gamma H,y_0 I_N)$. This characterisation has a natural counterpart for the marginal measures corresponding to the diagonal entries and the eigenvalues of the random matrix. This is summarised in the following theorem, in which we employ the standard notation $\E$ and $\mathbb{P}$. In particular, $\E[O]$ is the expectation value of an observable $O$ and $\mathbb{P}(\Sigma)$ is the probability of a measurable set $\Sigma$, i.e., it is the expectation value $\mathbb{P}(\Sigma)=\mathbb{E}[\chi_\Sigma]$ of the indicator function $\chi$ of the set $\Sigma$.

The following function is used in the expression of the characteristic function of stable distributions.
\begin{equation}\label{stable.vector}
	\nu_\alpha(u)=\begin{cases}
		\displaystyle|u|^\alpha\left(1-i\sgn(u)\,\tan\frac{\pi\alpha}{2}\right)=\frac{(-iu)^\alpha}{\cos(\pi\alpha/2)},&\quad 0<\alpha< 1\ {\rm or}\ 1<\alpha\le 2,\\
		\displaystyle|u|\left(1+\frac{2i}{\pi}\sgn(u)\,\log|u|\right)=\frac{2i u}{\pi}\log(-i u),&\quad \alpha=1,
	\end{cases}
\end{equation}
with $u\in\R$. The complex root $(.)^\alpha$ and the complex logarithm are their principle values with a branch cut along the negative real axis.

\begin{theorem}[Stable Distributions of Invariant Random Matrices]\label{p3.1}\

	Let $Y\in\Herm(N)$ be a random matrix and $y_\diag$ be its diagonal entries. Let $0<\alpha\le2$, $\gamma> 0$, $y_0\in\R$ be fixed constants and $H$ be a normalised spectral measure on $\mathbb{S}_{\Herm(N)}$, which has joint probability measures  $h_\diag$ of the diagonal entries and $h_\eig$ of the eigenvalues of the associated random matrix $R\in\mathbb{S}_{\Herm(N)}$. The following statements are equivalent:
	\begin{enumerate}
		\item $Y$ follows a ${\rm U}(N)$ invariant stable distribution $S(\alpha,\gamma H,Y_0)$ for some fixed matrix $Y_0$;
		\item $Y$ follows a stable distribution $S(\alpha,\gamma H,y_0 I_N)$ with $Y_0=y_0 I_N$ being a multiple of the identity matrix and $H$ being ${\rm U}(N)$ invariant, i.e., for any $U\in\U(N)$ it is
		\begin{equation}\label{inv.prop}
		H(\dv R)=H(U\dv R U^\dagger);
		\end{equation}
		\item $Y$ is ${\rm U}(N)$ invariant; $y_\diag$ has the following characteristic function
		\begin{equation}\label{3.1.8a}
			\mathbb E[ \exp (i y_\diag^\top s)]=\exp\left(-\gamma \int_{\|t\|\le 1} h_\diag(\dv t)\nu_\alpha(s^\top t)+iy_0 1_N^\top s\right)
		\end{equation}
		with $y_0$ being the same as in statement (2) and $1_N=(1,\ldots,1)^T\in\R^N$;
		\item $Y$ is ${\rm U}(N)$ invariant; and it has the following characteristic function
		\begin{equation}\label{3.1.10}
			\mathbb E [\exp (i \tr  YU\diag(s)U^\dagger)]=\exp\left(-\gamma c_N\int_{r\in\mathbb S^{N-1}}h_{\mathrm{eig}}(\dv r)\sum_{\rho\in S_N}\frac{\nu_\alpha(s^\top r_\rho)(s^\top r_\rho)^{N(N-1)/2}}{\Delta(s)\Delta(r_\rho)}+i y_0 1_N^\top s \right)
		\end{equation}
		for all $U\in \U(N)$, where $y_0 1_N$ is the same as in statement (2), $r_\rho$ is the permuted vector of $r$ with respect to the permutation $\rho\in S_N$, $\diag(s)=\diag(s_1,\ldots,s_N)$ is the standard embedding of the vector $s=(s_1,\ldots,s_N)^\top$ as a diagonal matrix and the constant $c_N$ is given by
		\begin{equation}
			c_N=\frac{\Gamma(\alpha+1)\prod_{j=0}^Nj!}{\Gamma(\alpha+1+N(N-1)/2)}.
		\end{equation}
	\end{enumerate}
\end{theorem}

The definitions and notations that are needed to understand this theorem fully can be found in subsection~\ref{s3.1} while the theorem is proven in subsection~\ref{s3.2}. We also discuss the implication for strictly stable random matrix ensembles in subsection~\ref{s3.1}, where the sum $Y_1+Y_2$ of two independent copies of the random matrix $Y$ only needs to be rescaled but not to be shifted to get again a copy of $Y$.

To illustrate and to get a feeling what Theorem~\ref{p3.1} means and implies we present three examples for the spectral measure $H$ in Sec.~\ref{s5}. The  elliptical invariant matrix ensembles are the first class of examples, see subsection~\ref{sec.ell}. Those can be traced back to the Gaussian elliptical ensemble in the same manner as it was done in~\cite{BCP2008,AFV2010,AV2008,AAV2009} where one averages over the variance of the corresponding Gaussian ensemble. The second example have a Dirac delta function as the spectral measure, discussed in subsection~\ref{sec:Dirac}. They represent the simplest embedding of univariate probability theory, namely a real random variable multiplied by the identity matrix. In subsection~\ref{sec:orbit}, we discuss the third class of ensembles which we have dubbed orbital measures as the corresponding spectral measure is supported only on the orbit created by the conjugate action of the unitary group on a fixed Hermitian matrix.

The generalised central limit theorems and especially the domains of attraction of the stable invariant random matrix ensembles are analysed in Sec.~\ref{s4}. We first review what is known from multivariate probability theory in subsection~\ref{sec:stable}. Especially, we extend its classification by an additional equivalent statement in Theorem~\ref{multi_stable_law} which has been useful in proving our third main theorem that relates the domain of attraction for random matrices, their eigenvalues and their diagonal entries.

\begin{theorem}[Domains of Attraction]\label{DoA}\

	Let $0<\alpha<2$ and $H$ be a normalised ${\rm U}(N)$ invariant spectral measure, see Eq.~\eqref{inv.prop}, with induced measures $h_\diag$ and $h_\eig$ for the diagonal entries and the eigenvalues, respectively. Let $X\in\Herm(N)$ be an invariant  random matrix, $x_{\eig}$ be its eigenvalues and $x_\diag$ be its diagonal entries. The following statements are equivalent:
	\begin{enumerate}
		\item $X$ belongs to the domain of attraction of $S(\alpha,H,0)$, meaning for $X_1,\ldots,X_m$ being $m$ independent copies of $X$ there exist two sequences of positive real numbers $\{B_m\}_{m\in\mathbb{N}}\subset\R_+$ and real numbers $\{A_m\}_{m\in\mathbb{N}}\subset\R$ such that the limit,
		\begin{equation}\label{conv_in_distr}
		\lim_{m\to\infty}\left(\frac{X_1+\ldots+X_m}{B_m}-A_mI_N\right)=Y
		\end{equation}
		is in distribution and the invariant random matrix $Y$ is drawn from $S(\alpha,H,0)$;
		\item $x_\diag$ belongs to the domain of attraction of $S(\alpha,g_\alpha,y_11_N)$ with $g_\alpha$ depending on $h_\diag$ as defined in~\eqref{g_alpha} and the same sequences $\{B_m\}_{m\in\mathbb{N}}\subset\R_+$ and real numbers $\{A_m\}_{m\in\mathbb{N}}\subset\R$ such that
		\begin{equation}\label{conv_in_distr.b}
		\lim_{m\to\infty}\left(\frac{x_{\diag,1}+\ldots+x_{\diag,m}}{B_m}-A_m1_N\right)=y_\diag
		\end{equation}
		in distribution, with $y_{\diag}$ drawn from $S(\alpha,g_\alpha,y_1 1_N)$, $1_N=(1,\ldots,1)$ and $y_1=0$ for $\alpha\neq1$ and $y_1$ given by~\eqref{3.1.12b} for $\alpha=1$;
		\item $x_{\rm eig}$ belongs to the domain of attraction of $S(\alpha,h_\eig,0)$ with  the same sequences $\{B_m\}_{m\in\mathbb{N}}\subset\R_+$ and the real numbers $\{A_m\}_{m\in\mathbb{N}}\subset\R$ and an additional sequence $\{C_m\}_{m\in\mathbb{N}}\subset\R_+$ such that
		\begin{equation}\label{conv_in_distr.c}
		\lim_{m\to\infty}C_m\left(\frac{x_{\eig,1}+\ldots+x_{\eig,m}}{B_m}-A_m1_N\right)=y
		\end{equation}
		in distribution, with $y$ drawn from $S(\alpha,h_\eig,0)$. If additionally $\tr  Y=\tr \diag(y_{\diag})=\tr \diag(y)\neq0$ with $\diag$ being the canonical embedding of a vector as a diagonal matrix, then it holds $C_m=1$ for all $m\in\mathbb{N}$.
	\end{enumerate}
\end{theorem}

The multi-dimensional analogue was firstly proved in~\cite{Rvaceva} and then restated in~\cite{Shimura} (see also the following Theorem~\ref{multi_stable_law} (1)$\Rightarrow$(2)). Our contribution is to give another equivalent statement (3) for classifying the domain of attraction, as we will make direct use of the latter to prove the analogue for Hermitian invariant ensembles.

Theorem~\ref{DoA} is proven in subsection~\ref{sec:stable} and only applies for a stability exponent $\alpha\in(0,2)$. We would like to point out that the additional statement in part (3) is born out of the fact that the trace of the right hand sides of~\eqref{conv_in_distr},~\eqref{conv_in_distr.b} and~\eqref{conv_in_distr.c} yields exactly the same real random variable. Thus, one only has to worry about the traceless part which cannot have a shift due to the invariance. Thus, the factor $C_m$ reflects the possibility that the traceless part can have a different scaling in $m$ than the trace which, however, can be fixed when the condition $\tr  Y=\tr \diag(y_{\diag})=\tr \diag(y)\neq0$ is satisfied.

 Additionally, we discuss the strict domains of convergences, meaning when $A_m=0$ in~\eqref{def.CLT}, in the same subsection. The Gaussian counterpart ($\alpha=2$) of this theorem is Theorem~\ref{DoA2} which is stated and proven in subsection~\ref{sec:Gauss}. It differs from the $\alpha<2$ case that there are not so many different stable random matrix ensembles that are invariant. Actually, it is only a three parameter family, namely the shift which must be proportional to the identity matrix, the variance of the trace of the matrix, and the variance of the traceless part of the random matrix.

In subsection~\ref{sec:heavy}, we relate the central limit theorem to the heaviness of the tail (the exponent with which the distribution algebraically decays) of those ensembles that lie in the domain of attraction of a stable invariant random matrix ensemble. In this subsection, we also explicitly define what a heavy-tailed random matrix is, see Definition~\ref{p3.2.2}, and prove equivalent definitions of the corresponding ``heaviness'' index.

 We conclude in Sec.~\ref{s6} and discuss the open problems that await when studying the asymptotic of large matrix dimensions.

\section{Characteristic Functions for Matrices}\label{s2}

\subsection{Invariant Ensembles and their Induced Measures}\label{s2.1}

Denote $\Herm(N)$ the set of all $N\times N$ Hermitian matrices for which we prove the ensuing theorems and propositions. We are sure that some of the statements can be readily extended to other matrix spaces without modifying much.  An invariant ensemble on $\Herm(N)$ is a random matrix $X\in\Herm(N)$, with a Borel probability measure $F$ on $\Herm(N)$ that is invariant under the action of unitary conjugations, i.e. $F(\dv X)=F(U\dv XU^\dagger)$ for all $U\in\U(N)$ and $X\in\Herm(N)$. This invariance has certain implications. For instance the distribution of the eigenvectors is always given by the normalised Haar measure on the co-set $\U(N)/\U^N(1)$, denoted by $\dv U$, regardless of the probability measure $F$. Hence, elements on the same orbit, say $\left\{UX_0U^\dagger: U\in\U(N) \right\}$ with $X_0\in\Herm(N)$, are uniformly distributed. The whole statistics is governed by a set of $N$ degrees of freedom. Those can be the eigenvalues or the diagonal entries of the random matrix.

To account for the derivatives in the relation between the eigenvalues and the diagonal entries of  the random matrix given by the derivative principle~\ref{p2.2}, we need to understand the Borel probability measure as a tempered distribution on $\Herm(N)$. Indeed, it is guaranteed that such a measure is a tempered distribution, see~\cite[Example 7.12(2)]{Rudin}. The corresponding test-functions are the Schwartz functions on $\Herm(N)$. We recall that Schwartz functions are smooth functions decaying faster than any polynomial. We denote the set of all Schwartz functions on an open subset $\Omega$ of $\R^d$ (for some $d$) by $\mathscr S(\Omega)$. In the present work it is always either $\Omega=\Herm(N)$ or $\Omega=\R^N$. The invariance of $F$ takes the following weak relation
\begin{equation}
	\int_{\Herm(N)} \Phi(X) F(\dv X)=\int_{\Herm(N)} \Phi(UXU^{\dagger}) F(\dv X),
\end{equation}
for any Schwartz function $\Phi\in \mathscr S(\Herm(N))$ and any unitary matrix $U\in\U(N)$. $U^\dagger$ is the Hermitian conjugate of $U$.

For our purposes we need two measures of $F$ which are to the marginal measures of the eigenvalues, $f_\eig$, and of the diagonal entries, $f_\diag$. Let $x=(x_1,\ldots,x_2)$ be the eigenvalues of $X\in\Herm(N)$, the Weyl decomposition formula~\cite[\S 1.3.4]{Fo10} tells us that $F$ only depends on $f_\eig$, which leads to the following defining relation
\begin{equation}\label{def_ePDF1}
	\int_{\Herm(N)}\Phi(X)F(\dv X)=:\int_{\R^N}\left(\int_{\U(N)/\U(1)}\Phi(U\diag(x)U^\dagger){\rm d}U\right)f_\eig(\dv x)=\int_{\R^N}\phi(x)f_\eig(\dv x)
\end{equation}
for any Schwartz function $\Phi\in \mathscr S(\Herm(N))$ which implies a permutation invariant Schwartz function $\phi(x)=\int_{\U(N)/\U(1)}\Phi(U\diag(x)U^\dagger){\rm d}U$ on $\R^N$. Certainly, this is also true without Weyl's decomposition formula but we need this factorisation $F(\dv X)={\rm d}U\,f_\eig(\dv x)$ later on. The map $x\mapsto\diag(x)$ is the embedding of the $N$ dimensional vector $x=(x_1,\ldots,x_N)$ as the diagonal matrix $\diag(x_1,\ldots,x_N)\in\Herm(N)$. The relation between $\Phi$ and $\phi$ naturally leads to a map $i_1$ between the symmetric Schwartz functions on $\R^N$ and the Schwartz functions invariant under the conjugate action of $\U(N)$ on $\Herm(N)$. Specifically, for $X=U\diag(x_1,\ldots x_N)U^{\dagger}$ with $U\in\U(N)$, we define
\begin{equation}\label{i_1}
i_1 \phi(X):=i_1 \phi\left(U\diag(x)U^{\dagger}\right)=\phi(x_1,\ldots,x_N)\qquad {\rm and}\qquad i_1^{-1}\Phi(x)=\Phi(\diag(x))
\end{equation}
for any $\U(N)$-invariant Schwartz function $\Phi\in\mathscr S(\Herm(N))$ and any $\phi\in i_1^{-1}\mathscr S(\Herm(N))\subset\mathscr S_{S_N}(\R^N)$.
Then~\eqref{def_ePDF1} can be rewritten as
\begin{equation}\label{def_ePDF2}
\int_{\R^N}\phi(x)f_\eig(\dv x)=\int_{\Herm(N)}i_1\phi(X)F(\dv X),
\end{equation}
which can be seen as the true definition of $f_\eig$. The reason why we need to restrict $\phi$ to the preimage $i_1^{-1}\mathscr S(\Herm(N))$ is that $i_1 \phi$ might be not smooth at the boundary of a Weyl chamber where the eigenvalues $x$ degenerate when choosing  an arbitrary symmetric Schwartz function $\phi\in\mathscr S_{S_N}(\R^N)$. A similar mechanism can be observed for the polar decomposition on $\R^2$ for rotation invariant functions in a radius and an angle. When choosing a Schwartz function for the radius it does not imply that the function is differentiable at the origin when canonically embedding the function as a rotation invariant function on $\R^2$.

In the case of the distribution of the diagonal entries we need a different embedding $i_2$ of $\R^N$ into $\Herm(N)$ which is symmetric in its entries but it is not invariant under the conjugation of an arbitrary unitary matrix. To this aim, we choose another arbitrary Schwartz function $\psi\in \mathscr S_{S_N}(\R^N)$ symmetric in its entries and define  the map $i_2$ from functions on $\R^N$ to functions on $\Herm(N)$ as follows
\begin{equation}\label{def.i2}
	i_2 \psi(X):=\psi(x_{11},\ldots,x_{NN}),
\end{equation}
where $x_{11},\ldots,x_{NN}$ are the diagonal entries of the matrix $X\in\Herm(N)$. We note that  $i_2 \psi(X)$ is still smooth but it is not rapidly decaying any more in the off-diagonal matrix entries. The diagonal entry distribution $f_\diag$ can be defined weakly as follows
\begin{equation}\label{def_diagPDF}
	\int_{\R^N}\psi(x)f_\diag(\dv x):=\int_{\Herm(N)} i_2 \psi(X)F(\dv X)
\end{equation}
for an arbitrary Schwartz function $\psi\in\mathscr{S}_{S_N}(\R^N)$.
The integral on the right hand side converges because $i_2\psi$ is bounded as $\psi$ is a Schwartz function on $\mathbb{R}^N$.

We remark here that both $i_1$ and $i_2$ can be extended to map continuous or bounded measurable functions to continuous or bounded measurable functions. For $i_1$, the boudnedness extends naturally and the measurablity entends because the spectral decomposition of the measure $\dv X$ allows $\phi(X)\dv X$ to be decomposed into the eigenvalue part and the eigenvector part, where the former is $i_1\phi(x)\dv x$. The continuity extends . For $i_2$, the existence of its extension can be seen easily as $i_2$ is only a restriction on $N$ variables. These extensions will be useful in the proof of Theorem 3.

Despite that $i_2\psi$ is not an invariant function under the conjugate action of $\U(N)$, one can create a related function which is invariant, because the measure $F$ is invariant. First we introduce a suitable notation. Let $B(\Omega)$ be the set of bounded measurable functions on $\Omega\subset \R^d$. Then for any bounded measurable function $g\in B(\Herm(N))$, we use the notation
\begin{equation}\label{am_functions}
	\la g(UXU^\dagger) \ra:=\int_{U\in\U(N)}g(UXU^\dagger)\dv U.
\end{equation}
The right hand side converges by the boundedness of $g$, and in fact $\la g(UXU^\dagger) \ra$ (in terms of a function of $X$) is also a bounded measurable function on $\Herm(N)$. The invariance of $\la g(UXU^\dagger) \ra$ under the conjugation of $V\in\U(N)$ is induced by the normalised Haar measure $\dv U$. It is simple to check that $\la g(UXU^\dagger) \ra$ is a Schwartz function when $g$ is one, too. The average~\eqref{am_functions} will be very helpful in the ensuing chapters.

In a weak sense we can extend the average~\eqref{am_functions} to probability measures. For instance, let $G(\dv X)$ be a probability measure on $\Herm(N)$ which is not necessarily invariant under the conjugate action of $\U(N)$. Then, for any bounded measurable function $g\in B(\Herm(N))$ we define
\begin{equation}\label{am_measures}
	\int_{\Herm(N)} g(X)\la G(U\dv X U^\dagger)\ra:=\int_{\Herm(N)} \la g(UXU^\dagger)\ra G(\dv X).
\end{equation}
The right hand side converges by the boundedness of $g(UXU^\dagger)$. Apparently, the construction guarantees the invariance of $\la G(U\dv X U^\dagger)\ra$ as it is traced back to the one of $\la g(UXU^\dagger)\ra$. Hence, $\la G(U\dv X U^\dagger)\ra$ defines a new invariant probability measure on $\Herm(N)$ in terms of $\dv X$.

The construction above allows us to understand $f_\diag$ in another way. Since $F$ is invariant, one can apply a conjugation of $U\in\U(N)$ on the right side of~\eqref{def_diagPDF} which keeps the integral the same. Thence, we can compute
\begin{equation}\label{2.2.3}
	\int_{\Herm(N)} i_2\psi(X)F(\dv X)=\int_{\Herm(N)} i_2\psi(X)\la F(U\dv XU^\dagger )\ra=\int_{\Herm(N)} \la i_2\psi(UXU^\dagger)\ra F(\dv X).
\end{equation}
Therefore, the flaw of the lack of  invariance of $i_2\phi$ can be resolved in this way. This insight will serve us well in the proof of Theorem~\ref{p2.2}.

\subsection{Fourier and spherical transforms}\label{s2.2}

Next we will briefly review the harmonic analysis on invariant ensembles in terms of distributions as it will be the main tool in our proofs. A common way to extend harmonic analysis from functions to distributions is through their actions on test functions, as we will proceed in this way. To begin with, we state our convention for Fourier transforms of Schwartz functions $\phi\in \mathscr S(\R^d)$.

\begin{definition}[Fourier Transform on $L^1(\R^N)$]\

	A Fourier transform of a Lebesgue integrable function $\phi\in L^1(\R^d)$ is given by
	\begin{equation}
		\f \phi(s):=\int_{\R^d}\phi(x)\exp \left(i\sum_{j=1}^Nx_js_j\right)\dv x,\quad s\in\R^d.
	\end{equation}
\end{definition}

The Fourier transform on the set of symmetric Schwartz functions $\mathscr S_{S_N}(\R^N)$ as well as of general Schwartz functions $\mathscr S(\R^d)$ is given by restriction to the corresponding subspaces.

	The Fourier transformation is an isomorphism on $\mathscr S(\R^d)$ with an inversion formula given by
	\begin{equation}
		\phi(x)=\f^{-1}\Big[\f\phi(s)\Big](x):=\frac{1}{(2\pi)^d}\int_{\R^d}\f\phi(s)\exp \left(-i\sum_{j=1}^dx_js_j\right)\dv s,
	\end{equation}
	and a differentiation formula given by
	\begin{equation}\label{Fourier_derivative}
		\f\Big[ (i\partial_{x_j})^k \phi(x)\Big](s)=s_j^k\f\phi(s),\quad k=0,1,\ldots
	\end{equation}
There is no need to introduce a regularisation for the inversion as the Fourier transform of a Schwartz function is a Schwartz function and, thus, the bijectivity holds.

The spherical transform is the generalisation of the concept of a Fourier transform on curved manifolds. This transform has to preserve the respective group invariance which is in our case the conjugate action of $\U(N)$ on $\Herm(N)$. There is a general framework of spherical transform introduced in~\cite{Helgason2000} on homogeneous space with a compact Lie group action, whose most general form is beyond the scope of the present work. We will instead restrict ourselves to the case of invariant ensembles on $\Herm(N)$. Since $\Herm(N)$ is an $N^2$-dimensional vector space, one can also directly derive the spherical transform from the corresponding Fourier transform on $\Herm(N)$.

Let $\phi\in \mathscr S_{S_N}(\mathbb{R}^N)$ and, thus, $i_1\phi\in L^1(\Herm(N))$ is an invariant $L^1$-function, recall the definition~\eqref{i_1} of $i_1$. Understanding Hermitian matrices as $N^2$-dimensional real vectors, the matrix Fourier transform of $i_1\phi$ is given by
\begin{equation}\label{2.2.4}
	\f_{\Herm(N)} i_1\phi(S):=\int_{\Herm(N)}i_1\phi(X)\exp (i\tr XS)\dv X,\quad S\in\Herm(N).
\end{equation}
Here we specify the fact that this is a matrix Fourier transform by the subscript $\Herm(N)$ as there is the difference of some factors of $2$ once one writes the inner product $\tr XS$ in terms of the real components of $X$ and $S$. When applying the spectral decomposition for $X=U\diag(x)U^\dagger$ the corresponding integrand $i_1\phi(X)\dv X$ factorises into $\phi(x)\dv x$ and the normalised Haar measure ${\rm d}U$ on $\U(N)/\U^N(1)$. The integral over $U$ can be readily extended to an integral over the whole group $\U(N)$ as $\U^N(1)$, embedded as diagonal matrices of complex phases, commutes with $\diag(x)$. Since the only $U$ dependence can be found in the exponential function, we arrive at Harish-Chandra–Itzykson-Zuber (HCIZ) integral~\cite{Harish-Chandra,Itzykson-Zuber}
\begin{equation}\label{2.2.5}
	K(x,s):=\int_{\U(N)}\exp\left(i\tr  U\diag(x)U^\dagger \diag(s)\right)\dv U=\prod_{j=0}^{N-1}j!\ \frac{\det[e^{ix_js_k}]_{j,k=1}^N}{\Delta(x)\Delta(s)}.
\end{equation} 
the function $\Delta(x):=\prod_{1\le j<k\le N} (x_k-x_j)$ is the Vandermonde determinant. Later, we also need an analogous notation $\Delta(\partial_x):=\prod_{1\le j<k\le N} (\partial_{x_k}-\partial_{x_j})$ which is the Vandermonde determinant of the partial derivatives.

With the help of the integral~\eqref{2.2.5}, we define the following spherical transform.

%\newpage

\begin{definition}[Spherical Transform on $\mathscr S(\R^d)$]\

	Let $\phi/\Delta^2\in \mathscr S_{S_N}(\R^N)$ be a symmetric function. The spherical transform $\s: \Delta^2\mathscr S_{S_N}(\R^N)\mapsto \mathscr S_{S_N}(\R^N)$ of $\phi$ and its inverse are given by
	\begin{equation}\label{2.2.6}
		\s \phi(s):=\frac{\prod_{j=0}^Nj!}{\pi^{N(N-1)/2}}i_1^{-1}\f_{\Herm(N)} i_1 \frac{\phi}{\Delta^2}(s),\quad \phi(x):=\frac{\pi^{N(N-1)/2}}{\prod_{j=0}^Nj!}\Delta^2(x)i_1^{-1}\f^{-1}_{\Herm(N)} i_1  \s\phi(x),\quad x,s\in\R^N,
	\end{equation}
	where $[\phi/\Delta^2](x)=\phi(x)/\Delta^2(x)$ and similarly elsewhere where $\Delta^2$ appears.
\end{definition}

This definition looks a bit involved with the function $\phi/\Delta^2$ because we want to stick with the notation of~\cite{ZK2020}. This will help the reader to identify the results from this work which are applied in the present one.

	Note that \eqref{2.2.6} is well defined as $i_1$ is a bijection between invariant functions on $\Herm(N)$ and symmetric functions of the eigenvalues as the normalised Haar measure of $\U(N)$ is unique. An explicit integral expression of $\s f$ is given by
	\begin{equation}
		\s \phi(s)=\int_{\R^N}\phi(x)K(x,s)\dv x,\quad s\in\R^N,
	\end{equation}
	with $K(x,s)$ given in~\eqref{2.2.5}, and its inversion formula has the form
	\begin{equation}
		\phi(x)=\s^{-1}\Big[\s\phi(s)\Big](x):=\frac{(-1)^{N(N-1)/2}}{\prod_{j=0}^{N}(j!)^2}\Delta(x)^2\int_{\R^N}\s\phi(s)\,K(x,-s)\,\Delta(s)^2\,\frac{\dv s}{(2\pi)^N}
	\end{equation}	
	when $\phi/\Delta^2\in \mathscr S_{S_N}(\R^N)$ is symmetric. These are the formulas we have employed as definitions in~\cite{ZK2020} and will be used throughout the present work. We would like to highlight when $\psi=\phi/\Delta^2\in \mathscr S_{S_N}(\R^N)$ is a Schwartz function on $\mathbb{R}^N$ so is $\mathcal{S}\phi=\mathcal{S}(\Delta^2\psi)$ since $K(x,s)$ is entire  in all of its $2N$ arguments and its derivatives maximally grow polynomially for real arguments. Similarly $[\mathcal{S}^{-1}\psi]/\Delta^2$ is a Schwartz function as it is essentially the same formula as $\mathcal{S}(\Delta^2\psi)$ apart from a reflection at the origin and a normalisation constant.

With the help of spherical and Fourier transforms of a test function $\psi\in \mathscr{S}_{S_N}(\R^N)$ one can easily define those for tempered distributions in a weak sense.

\begin{definition}[Fourier and Spherical Transform on Tempered Distributions]\

	For every Schwartz function $\psi\in \mathscr S_{S_N}(\R^N)$, the Fourier and spherical transforms of a tempered distribution $f$ on $\R^N$, that is symmetric in its entries, are given by
	\begin{equation}\label{2.2.9}
		\int _{\R^N}\Delta^2(s)\psi(s)\s f(\dv s):=\int _{\R^N} \s[\Delta^2\psi](x) f(\dv x),
	\end{equation}
	\begin{equation}\label{2.2.10}
		\int _{\R^N} \psi(s)\f f(\dv s):=\int _{\R^N} \f\psi(x) f(\dv x).
	\end{equation}
\end{definition}

The definition reflects Plancherel's theorem which is essentially enforced in this way. Indeed, when noticing that for instance the Fourier transform is up to a normalisation a unitary operator and we understand the integral over the measure $f$ as a dual vector (linear functional) and the Schwartz function $\psi$ as a vector, we see that Eq.~\eqref{2.2.10}  can be  equivalently written as $\int _{\R^N} \psi(s) f(\dv s)=\int _{\R^N} \f^{-1}\psi(x) \f f(\dv x)$ which is Plancherel's theorem. Moreover, because both $\s[\Delta^2.]$ and $\f_{\R^N}$ are isomorphisms of Schwartz functions, one can easily see that both $\Delta^2\s f$ and $\f f$ are well-defined tempered distributions.

\subsection{Derivative principle}\label{sec:der}

As mentioned before, the derivative principle relates the joint probability distributions of the eigenvalues and of the diagonal entries for an invariant ensemble. In particular it shows that the former is equal to a linear differential operator acting on the latter. It was firstly stated in~\cite{Faraut} and then rediscovered and got its name in the studies of quantum information theory (see e.g. \cite{CDKW14,MZB16}). These are summarised and stated as an existence and uniqueness theorem in~\cite[Cor 3.4]{ZK2020}, with an analytical conditions of $F$ assumed for technical purposes.

To recall the proposition given in~\cite[Corollary~3.4]{ZK2020}, we need to introduce the set
\begin{equation}\label{tilde_L}
	\tilde L^1(\R^N):=\bigg\{g\in L^1(\R^N): \int_{\R^N}|x^a\partial_x^b\, g(x)|\dv x<\infty,\forall a,b\in\mathbb N_0^N\text{ and }|a|,|b|\le\frac{N(N-1)}{2}\bigg\},
\end{equation}
where $L^1(\R^N)$ are the Lebesgue integrable functions on $\R^N$.
Here $a, b$ are multi-indices and we employed the notation $x^a=x_1^{a_1}\ldots x_N^{a_N}$ and $|a|=a_1+\ldots+a_N$. Identifying Schwartz functions as Lebesgue functions one can readily check that $\mathscr S(\R^N)\subset \tilde L^1(\R^N)$. Certainly, the Fourier transform and its inverse exist for any function in the set~\eqref{tilde_L}, and so do their derivatives up to order $N(N-1)/2$, namely via the formula~\eqref{Fourier_derivative}. Once again we denote the subset of symmetric functions by $L^1_{S_N}(\R^N)$ and $\tilde L^1_{S_N}(\R^N)$.

\begin{proposition}[Derivative Principle for Densities~\cite{Faraut,CDKW14,MZB16,ZK2020}]\label{p2.3.1}\

	Let $X\in \Herm(N)$ be a $\U(N)$-invariant matrix with the joint probability density of
	\begin{equation}\label{p2.3.1.cond}
		 {\rm the\ eigenvalues}\ f_\eig\in L_{S_N}^1(\R^N)\ {\rm and\ of\ the\ diagonal\ entries}\ f_\diag\in \tilde L_{S_N}^1(\R^N).
\end{equation}
 Then there exists a unique symmetric function $w\in \tilde L^1(\R^N)$, such that
	\begin{equation}\label{2.3.2}
		f_\eig(x)=\frac{1}{\prod_{j=0}^Nj!}\Delta(x)\Delta(-\partial_x)w(x),
	\end{equation}
	for almost all $x\in\R^N$. Moreover one has the relations
	\begin{equation}\label{2.3.3}
		\s f_\eig(s)=\f w(s)
	\end{equation} 
	and
	\begin{equation}
		w(x)=f_\diag(x)
	\end{equation}
	for almost all $x,s\in\R^N$.
\end{proposition}

On the one hand, this proposition is a powerful tool for analysing sums of invariant ensembles on $\Herm(N)$. On the other hand, it gives a unique characterisation of such ensembles in terms of their diagonal entry distributions which is advantageous compared to the eigenvalues. We would like to illustrate this. Let $A,B,C\in\Herm(N)$ be invariant with eigenvalue distribution $f_A, f_B, f_C$ and diagonal entry distribution $f_{\diag}^{(A)}, f_{\diag}^{(B)}, f_{\diag}^{(C)}$, satisfying the conditions in Proposition~\ref{p2.3.1}. Assume $A+B=C$, then their diagonal entries also add component-wise in this sum, so that they satisfy the multivariate convolution relation
	\begin{equation}
	f_{\diag}^{(A)}\ast f_{\diag}^{(B)} = f_{\diag}^{(C)},
	\end{equation}
	This relation reduces a sum of two independent random matrices into a sum of two independent random vectors (the diagonal entries), and by applying~\eqref{2.3.2} with $w=f_{\diag}$ the eigenvalue distributions can be recovered. This methodology is very helpful when considering central limit theorems for invariant ensembles as it can, thereby, be reduced to the ordinary central limit theorem for random vectors in $\R^N$.

Theorem~\ref{p2.2} is the natural generalisation of Proposition~\ref{p2.3.1} to distributions, meaning we relax the condition~\eqref{p2.3.1.cond}. Then, we do not need to rely on a density function of the probability measure which is crucial for our purposes. The difference in $w$ and $f_{\diag}$ for distributions lies in the fact that the agreement is only weakly. This means they agree after integrating over a suitable test function which are in our case Schwartz functions that are permutation invariant in their $N$ arguments.

\begin{proof}[Proof of Theorem~\ref{p2.2}]
	We first prove the existence of $w$, i.e., Eq.~\eqref{2.3.5} with $w=f_\diag$. Along a derivation similar to the proof of~\cite[Proposition 3.1]{ZK2020} one has
	\begin{equation}\label{2.3.7}
		\frac{1}{\prod_{j=0}^{N}j!}\Delta(\partial_x)\Big(\Delta(x)\phi(x)\Big)=\f\s^{-1} \phi
	\end{equation}
	for any symmetric Schwartz function $\phi$ on $\R^N$.
	Therefore the right hand side of~\eqref{2.3.5} can be written as
	\begin{align}
		\int \f\s^{-1}\phi(x)\,f_\diag(\dv x) &= \int_{\Herm(N)} i_2 \f \s^{-1} \phi(X) F(\dv X)=\int_{\Herm(N)} \la i_2 \f \s^{-1} \phi(UXU^{\dagger})\ra F(\dv X),
	\end{align}
	recalling the definition~\eqref{def.i2} of the map $i_2$.
	The second line is obtained with the help of Eq.~\eqref{2.2.3}.
	
	After writing $\f$ explicitly, we have
	\begin{equation}
	\begin{split}
	\int \f\s^{-1}\phi(x)\,f_\diag(\dv x)
	=&\int_{\Herm(N)} \la\int_{\R^N} \exp[i\tr UXU^{\dagger}\diag(s)]\s^{-1} \phi(s)\,\dv s\ra F(\dv X)\\
	=&\int_{\Herm(N)} \int_{\R^N} \la\exp[i\tr UXU^{\dagger}\diag(s)]\ra\s^{-1} \phi(s)\,\dv s\, F(\dv X)\\
	=&\int_{\Herm(N)} i_1\s\s^{-1} \phi(X) F(\dv X)\\
	=&\int_{\Herm(N)} i_1\phi(X) F(\dv X),
	\end{split}
	\end{equation}
	and this is exactly equal to the left hand side of~\eqref{2.3.5}; recalling the definition~\eqref{def_ePDF2}. Making use of~\eqref{2.3.7} and replacing $\phi$ by $\s \phi$ in~\eqref{2.3.5} with $\phi/\Delta^2\in\mathscr{S}_{S_N}(\mathbb{R}^N)$, Eq.~\eqref{2.3.3} can then be obtained from the bijectivity of $\s: \Delta^2\mathscr S_{S_N}(\R^N)\mapsto \mathscr S_{S_N}(\R^N)$.
	
	For the uniqueness of $w$ in distribution, it is equivalent to show that
	\begin{equation}
		\int_{\R^N}\f\s^{-1} \psi(x)u(\dv x)=0,\quad \forall \psi\in\mathscr S_{S_N}(\R^N)\Rightarrow u(\dv x) = 0,\ {\rm meaning}\ u\ {\rm is\ the\ constant\ zero\ measure},
	\end{equation}
	for all Borel measures $u$. This is a simple corollary of the fact that $\s^{-1}: \mathscr S_{S_N}(\R^N)\to\Delta^2\mathscr S_{S_N}(\R^N)$ and $\f:\Delta^2\mathscr S_{S_N}(\R^N)\to\Delta^2(\partial_x)\mathscr S_{S_N}(\R^N)$ are injections. The latter is clear when writing 
	\begin{equation}
	\mathcal{F}(\Delta^2\phi)(x)=\Delta^2(\partial_x)\mathcal{F}\phi(x)
	\end{equation}
	and using the fact that a Fourier transform of a Schwartz function is a Schwartz function.
	
\end{proof}

In a probability theoretic context, the Fourier transform is referred to as the characteristic function. One notices the integrals in~\eqref{2.2.9} and~\eqref{2.2.10} are absolutely integrable.  Fubini's theorem can be applied to obtain alternative expressions for the two transforms
\begin{equation}
\s f(s)=\mathbb E[ K(x,s)],\quad
\f f(s)=\mathbb E [\exp(ix^\top s)],
\end{equation}
where the expectation is taken with respect to the real random vector $x\in\R^N$ with measure $f(dx)$  and $x^T$ is the transpose of the vector $x$. Therefore, for the two random Hermitian matrices $X$ and $S$ with eigenvalues $x$ and $s$ and diagonal entries $x_\diag$, Theorem~\ref{p2.2} can be rewritten as
\begin{equation}
	\E[ \exp (i\tr XS)]=\E[K(x,s)]=\E[ \exp (ix_\diag^\top s)].
\end{equation}
Hereafter, we will refer any of the three expression as the characteristic function of $X$ as they all agree, evidently.

\section{Stable Invariant Ensembles}\label{s3}

\subsection{Identifications of Stable Invariant Ensembles}\label{s3.1}

A random vector $y\in\R^N$ is called stable when the sum of two statistically independent copies of this vector $y_1+y_2$ is equal in distribution to $\sigma\, y+\delta y$ with $\sigma>0$ a rescaling and $\delta y\in\R^N$ a non-random shift.
A full classification of stable random vectors is given in terms of their characteristic functions. A random vector $y\in\R^d$ is stable if and only if its characteristic function is given by~\cite{Rvaceva}
\begin{equation}\label{3.1.1}
	\mathbb E[\exp( i y^\top s)]=\exp\left(-\gamma\int_{r\in\mathbb S^{d-1}}h(\dv r)\nu_\alpha(s^\top r)+i y_0^\top s\right),
\end{equation}
where $\gamma\ge 0$ is a scaling of the distribution, $h$ is a Borel probability measure on the unit sphere $\mathbb S^{d-1}$ in $\R^d$, called the \textit{spectral measure}, and $y_0\in\R^d$ is a fixed vector acting as a shift. The index $\alpha\in(0,2]$ is the stability exponent and tells us  how fast the distribution drops off at infinity. The spectral measure $h(\dv r)$ describes the anisotropy of the distribution and, thence, plays the role of the asymmetry index in univariate probability theory, usually denoted by $\beta$ for L\'evy distributions which should not be confused with the Dyson index in random matrix theory. The function $\nu_\alpha$ is explicitly given by~\eqref{stable.vector}. We denote such a stable distribution by $S(\alpha,\gamma h,y_0)$. The combination $\gamma h$ is born out of the fact that these two parts of the distribution come always in this combination.

For $\alpha=2$ we obtain the Gaussian case where $v_0$ is equal to the mean of the random vector and the integral  with the measure $\gamma h$ creates the corresponding covariance matrix. Another trivial case is the Dirac delta distribution that can be readily regained by setting $\gamma=0$. Equation~\eqref{3.1.1} then only contains a fixed shift, which corresponds to a fixed vector $y=y_0$. As this latter case is less interesting, hereafter we always consider $\gamma>0$.

A particular subset of stable distributions we would like to point out are the \textit{strictly stable distributions}. They satisfy the condition that the sum of two copies of the random vector $y$ is in distribution equal to the rescaling of the very same random vector without adding any additional fixed shift. If $\alpha\ne 1$, this is achieved when the shift vector $y_0=0$. However when $\alpha=1$,  the spectral measure must satisfy the vector-valued equation
\begin{equation}\label{stable.alpha1}
	\int_{\|r\|^2=1} h(\dv r)r=0
\end{equation}
with $\|r\|=\sqrt{\sum_{j=1}^d r_j^2}$ being the Euclidean norm on $\R^d$.
This equation results from the logarithm which satisfies the rescaling property $\log(\sigma |s^Tr|)=\log|s^Tr|+\log(\sigma)$ for any $\sigma>0$. We underline that $y_0$ can be indeed arbitrary in this case to get a strictly stable distribution.
For more details, we refer the reader to the textbook~\cite{ST94}.

To carry the construction and classification of stable distributions over to random matrices, we identify $\Herm(N)$ with $\R^{N^2}$.

\begin{definition}[Stable and Strictly Stable Random Matrix Ensembles]\label{def:stableRM}\

A random matrix $X\in\Herm(N)$ is \textit{stable} when for the sum $X_1+X_2$ of two random matrices $X_1, X_2\in\Herm(N)$ independently drawn and identical in distribution to $X$ there is a fixed scaling $\sigma>0$ and a fixed matrix $X_0\in\Herm(N)$ such that $X_1+X_2$ is equal in distribution to $\sigma X+X_0$. The random matrix $X$ is called \textit{strictly stable} when $X_0=0$.
\end{definition}

The inner product $s^Tr$ is in the current case given by the  Hilbert-Schmidt inner product $(S,R)\mapsto\tr SR$ with $S,R\in\Herm(N)$. Especially, the norm is then $\sqrt{\tr R^2}$.  Then, the analogue of~\eqref{stable.vector}  for Hermitian random matrices which are stable is given by
\begin{equation}\label{3.1.3}
	\E [\exp(i \tr   YS)]=\exp\left(-\gamma\int_{\mathbb{S}_{\Herm(N)}}H(\dv R)\,\nu_\alpha(\tr   SR)+i\tr  Y_0S \right),
\end{equation}
with $X_0$ being a fixed Hermitian matrix and $H$ being a probability measure on the compact set $\mathbb{S}_{\Herm(N)}:=\{R\in\Herm(N)|\tr   R^2=1 \}\simeq \mathbb S^{N^2-1}$, meaning the Frobenius norm defines this kind of sphere. For convenience we denote ensembles which have the characteristic function~\eqref{3.1.3} as $S(\alpha,\gamma H,Y_0)$. The spectral measure $H$ is a Borel probability measure on $\mathbb{S}_{\Herm(N)}$.

So far there is nothing new; what we wrote as the identification of $\Herm(N)$ with $\R^{N^2}$ is a triviality. Indeed, the theory of Wigner matrices~\cite{CB1994,BJNPZ2007,BG2008,TBT2016,AM2001,PSFG2010} with independent L\'evy-stable distributions as entries is covered by the above definition of stability. The choice of the spectral measure in~\eqref{3.1.3} is given by a sum of weighted Dirac delta functions on the standard real basis of $\Herm(N)$, i.e.,
\begin{equation}
\int_{\mathbb{S}_{\Herm(N)}}H(\dv R) \Phi(R)=\sum_{a=1}^N[d_+ \Phi(e_a)+d_-\Phi(-e_a)]+2o\sum_{1\leq a<b\leq N}\int_0^{2\pi}A(\dv \varphi) \Phi(e_{a,b}(\varphi))
\end{equation}
for any test function $\Phi\in\mathscr S(\mathbb{S}_{\Herm(N)})$. The constants $d_+,d_-,o\geq0$ and $N(d_++d_-)+N(N-1)o=1$ are fixed and the two sets of Hermitian matrices are given by $\{e_a\}_{m,n}=\delta_{m,a}\delta_{n,a}$ and $\{e_{a,b}(\varphi)\}_{m,n}=(e^{i\varphi}\delta_{m,a}\delta_{n,b}+e^{-i\varphi}\delta_{m,b}\delta_{n,a})/\sqrt{2}$ where $\delta_{a,b}$ is the Kronecker delta. The measure $A(\dv \varphi)$ is a Borel probability measure on $[0,2\pi)$ and describes how anisotropic the off-diagonal element are distributed.

The novel idea of the present work is to combine the stability condition with the invariance under the conjugate action of the unitary group $\U(N)$. Then the eigenvectors are Haar distributed regardless which stable measure one has chosen. Thus we would like to remove this part of the information from the spectral measure $H(\dv R)$.  We have summarise our findings in the Theorem~\ref{p3.1} which will be proven in subsection~\ref{s3.2}. Statement (3) of this theorem can be further interpreted that $y_\diag$ has a stable distribution, too. This is trivial, as the marginal distributions of stable distribution must be stable under matrix addition. The following theorem states the exact relationship between the spectral measures of stable ensembles and their diagonal stable distributions. It is also proven in subsection~\ref{s3.2}. We recall that the measure $h_{\rm diag}$ in the following theorem is related to the spectral measure $H$ as given in~\eqref{def_diagPDF}.

\begin{theorem}[Identification with Stable Distributions in $\R^N$]\label{p3.1a}\

	We employ the notations of Theorem~\ref{p3.1} and $1_N=(1,\ldots,1)\in\R^N$. The statements of Theorem~\ref{p3.1} are equivalent to the respective claims:
\begin{enumerate}
\item	When $\alpha\ne 1$, $y_\diag$ follows a stable distribution $S(\alpha,\gamma g_\alpha,y_0 1_N)$ with its spectral measure $g_\alpha$ given by
	\begin{equation}\label{g_alpha}
		\int_{r\in\mathbb S^{N-1}} g_\alpha(\dv r)\psi(r):=\int_{0<\|t\|\le 1}h_\diag(\dv t)\,\|t\|^{\alpha}\psi\left(\frac{t}{\|t\|}\right)
	\end{equation}
	for any bounded measurable function $\psi\in B(\mathbb S^{N-1})$.
\item	When $\alpha=1$, $y_\diag$ instead follows $S(1,\gamma g_1,(y_0+y_1)1_N)$ with $g_1$ given by~\eqref{g_alpha} and $y_1$ satisfying
	\begin{equation}\label{3.1.12b}
		y_1 1_N=\frac{2}{\pi}\int_{0<\|t\|\le 1} h_\diag(\dv t)\,t\log \|t\|
	\end{equation}
\end{enumerate}	
	 
\end{theorem}

Equation~\eqref{3.1.12b} is essentially a consequence that the unitary invariance  implies  the permutation invariance of $h_\diag$. Therefore the right hand side of~\eqref{3.1.12b} must be the same for each entry of the vector. One can also check that $g_\alpha$ is a positive measure, as the right hand side of~\eqref{g_alpha} is positive for any bounded positive function $\psi$. However, $g_\alpha$ may not be normalised due to the additional factor $\|t\|^\alpha$. In particular it would be only normalised when $h_\diag$ is concentrated on the sphere $\mathbb S^{N-1}$, meaning $g_\alpha$ and $h_\diag$ are the same in distribution.

A strictly stable distribution is a special case where adding two of these only rescales it. As already stated, when the stability index $\alpha\ne 1$ we only need to set $y_0=0$. This is evidently true when we write the characteristic function in terms of the matrix $S\in\Herm(N)$, its diagonal entries $s_{\diag}$ or its eigenvalues $s_{\rm eig}$. However, the case $\alpha=1$ is different. For random matrices the condition~\eqref{stable.alpha1} for the spectral measure reads in terms of the matrix-valued equation
\begin{equation}\label{3.1.10a}
\int_{\mathbb{S}_{\Herm(N)}} H(\dv R)R=0.
\end{equation}
For the respective induced measures $h_\eig$ and $h_\diag$ this condition reads as follows which we prove immediately.

\begin{proposition}[Strictly Stable Ensembles with $\alpha=1$]\label{p3.2}
	We use the setting in Theorem~\ref{p3.1} with $\alpha=1$ and $Y$ be the stable and invariant random matrix. The following three statements are equivalent:
	\begin{enumerate}
		\item $Y$ is strictly stable, i.e., it satisfies \eqref{3.1.10a};
		\item the spectral measure for the diagonal entries  $h_\diag$ satisfies
		\begin{equation}\label{3.1.11a}
		\int_{\|t\|\le 1} h_\diag(\dv t)t=0;
		\end{equation}
		\item the spectral measure for the eigenvalues $h_\eig$ satisfies
		\begin{equation}\label{3.1.11}
		\int_{r\in\mathbb S^{N-1}} h_\eig(\dv r)r=0.
		\end{equation}
	\end{enumerate}
\end{proposition}

\begin{proof}[Proof of Proposition~\ref{p3.2}]
	Statement (1)$\Rightarrow$(2) is a triviality as one only needs to integrate over the off-diagonal entries of the matrix $R$ to obtain~\eqref{3.1.11a}.
	
	For (2)$\Rightarrow$(3), we combine the integration domains $\mathbb{S}^{N-1}\subset\R^N$ into the measure by defining
	\begin{equation}
	\int_{\R^N}\varphi(r)\hat{h}_\eig(\dv r):=\int_{\mathbb{S}^{N-1}}\varphi(r)h_\eig(\dv r)\ {\rm and}\ \int_{\R^N}\varphi(r)\hat{h}_\diag(\dv r):=\int_{\|r\|\leq 1}\varphi(r)h_\diag(\dv r)
	\end{equation}
	for an arbitrary bounded measurable function $\varphi$ on $\R^N$. What we want to do is to apply the derivative principle. The problem is that the components $t_j$ are neither invariant under permutation nor are they Schwartz functions on $\R^N$. Hence, we consider the Schwartz function $\varphi(r)=e^{-\epsilon\|r\|^2}\tr \diag(r)$ for an arbitrary $\epsilon>0$.
	Indeed, we have for any $j=1,\ldots,N$
	\begin{equation}
	\int_{r\in\mathbb S^{N-1}} h_\eig(\dv r)r_j=\frac{e^{\epsilon}}{N}\int_{\R^N} \hat{h}_\eig(\dv r)e^{-\epsilon\|r\|^2}\tr \diag(r)
	\end{equation}
	 due to the invariance of the remaining parts of the integral under permutations of the eigenvalues. Now we are prepared to apply the derivative principle~\eqref{2.3.5} yielding
	\begin{equation}
	\int_{r\in\mathbb S^{N-1}} h_\eig(\dv r)r_j=\frac{e^{\epsilon}}{N\prod_{j=0}^{N}j!}\int_{\|r\|\le 1} h_\diag(\dv r)\Delta(-\partial_r)\left[\Delta(r)e^{-\epsilon\|r\|^2}\tr \diag(r)\right].
	\end{equation}
	As $\epsilon>0$ has been arbitrary we can take the limit $\epsilon\to0$. This limit commutes with the integral as the integration is over a bounded set, and the integrand is $\Delta(-\partial_r)\left[\Delta(r)e^{-\epsilon\|r\|^2}\tr \diag(r)\right]<c<\infty$ for any $\|r\|\leq 1$ with an $\epsilon$ independent constant. This leads us to
	\begin{equation}
	\int_{r\in\mathbb S^{N-1}} h_\eig(\dv r)r_j=\frac{1}{N\prod_{j=0}^{N}j!}\int_{\|r\|\le 1} h_\diag(\dv r)\Delta(-\partial_r)\left[\Delta(r)\tr \diag(r)\right].
	\end{equation}
	When using
	\begin{equation}
	\Delta(-\partial_r)=\det[(-\partial_{r_a})^{b-1}]_{a,b=1,\ldots, N} \quad {\rm and} \quad \Delta(r)\tr \diag(r)=\det\left[\begin{array}{c|c} r_a^{b-1} & r_a^N \end{array}\right]_{\substack{a=1,\ldots,N\\b=1,\ldots,N-1}},
	\end{equation}
	where the latter means that the first $N-1$ columns are given by the matrix entries $r_a^0,\ldots, r_a^{N-2}$ and the last column by $r_a^N$,
	and the Laplace expansion of both determinants, one can readily check 
	\begin{equation}
	\Delta(-\partial_r)[\Delta(r)\tr \diag(r)]=\prod_{j=0}^{N}j!\ \tr \diag(r).
	\end{equation}
	This means we are left with
	\begin{equation}
	\int_{r\in\mathbb S^{N-1}} h_\eig(\dv r)r_j=\frac{1}{N}\int_{\|r\|\le 1} h_\diag(\dv r)\tr \diag(r)=0
	\end{equation}
	because of~\eqref{3.1.11a} for any $j=1,\ldots,N$. This is equivalent to~\eqref{3.1.11}.
	
	Finally for (3)$\Rightarrow$(1),  we take the map $\diag$ of~\eqref{3.1.11} so that it becomes a diagonal matrix and then multiply from the left with $U\in\U(N)$ and from the right with $U^\dagger$ and then integrate over $U$ with respect to the normalised Haar measure $\dv U$ on $\U(N)/\U^N(1)$. Since $H(\dv R)=h_{\rm eig}(\dv r)\dv U$ with $R=U\diag(r)U^\dagger$, we obtain~\eqref{3.1.10a}. This finishes the proof.
\end{proof}

\subsection{Proofs of Theorem~\ref{p3.1} and~\ref{p3.1a}}\label{s3.2}

The idea of proving Theorem~\ref{p3.1} is by making use of the invariance and the averages~\eqref{am_functions} and~\eqref{am_measures}. As a preparation we are going to recall the following lemma from a well-known Fourier-Laplace transform of the modulus.

\begin{lemma}[]\label{p3.1.1} Let $\sgn$ be the signum function. For $x\in\R$ and $a>0$, one has the following identities.
	\begin{equation}\label{3.1.5}
		|x|^a=\frac{\Gamma(a+1)}{2\pi i}\int_{1+i\R}\frac{e^{zx}+e^{-zx}}{z^{a+1}}\dv z,
	\end{equation}
	\begin{equation}\label{3.1.6}
		|x|^a\sgn(x)=\frac{\Gamma(a+1)}{2\pi i}\int_{1+i\R}\frac{e^{zx}-e^{-zx}}{z^{a+1}}\dv z.
	\end{equation}
	The contour of both integrals is parallel to the imaginary axis starting at $1-i\infty$ and ending at $1+i\infty$ and the complex root $z^a$ is the principle value with the branch cut along the negative real line.
\end{lemma}
\begin{proof}
	After writing $|x|^a=x^a\chi_{x>0}+(-x)^a\chi_{x<0}$ and $|x|^a\sgn(x)=x^a\chi_{x>0}-(-x)^a\chi_{x<0}$ with $\chi$ being the indicator function, it is immediate that both identities are resulting from
	\begin{equation}\label{3.1.8}
		\frac{\Gamma(a+1)}{2\pi i}\int_{1+i\R}\frac{e^{zx}\dv z}{z^{a+1}}=x^a \chi_{x>0},
	\end{equation}
	for all $x\in\R$ and $a>0$. This is a well-known Fourier-Laplace transform which can be proven via complex analysis. For $x=0$ the integral vanishes because the integrand drops off fast enough.  In the case that $x<0$ we close the contour around the half axis $[1,\infty)$ whose interior is free of any singularities and, thence, vanishes due to Cauchy's theorem. When $x>0$ we can rescale $z\rightarrow z/x$ which is gives a the contour over the line $1/x+i\mathbb{R}$. One can shift the contour back to $1+i\mathbb{R}$ as no singularity is crossed. Finally we use the identity
	\begin{equation}
		\frac{\Gamma(a+1)}{2\pi i}\int_{1+i\R}\frac{e^{z}\dv z}{z^{a+1}}=1,
	\end{equation}
	which can be derived by closing the contour around the negative real axis. This finishes the proof.
\end{proof}

\begin{proof}[Proof of Theorem~\ref{p3.1}]
	
	Let us firstly show the equivalence of statements (1) and (2). In fact, (2)$\Rightarrow$(1) is trivial since the invariance is given in the statement. Thence, we only need to show its converse.
	
	Let us denote the exponent of the characteristic function by 
	\begin{equation}
	\xi(S)=\log\left[\mathbb{E}(e^{i \tr SX})\right],
	\end{equation}
	where the logarithm may experience a discontinuity when $\mathbb{E}(e^{i \tr SX})$ crosses the negative real line. This exponent is related to the function
	\begin{equation}
		\xi_0(S)=-\gamma \int_{\mathbb S_{\Herm(N)}}H(\mathrm d R)\nu_\alpha(\tr RS)+i\tr Y_0S.
	\end{equation}
	like $\xi(S)=\xi_0(S)+2\pi i k(S) $ with $k(S)\in\Z$. The invariance under the conjugate action of $\U(N)$ tells us that the characteristic function is unchanged under the conjugate action, too, so that $\xi(USU^\dagger)=\xi(S)$ for all fixed $U\in\U(N)$,
	 which implies
	\begin{equation}\label{diff_phi}
		\xi_0(USU^\dagger)-\xi_0(S)=2\pi n i
	\end{equation}
	with $n=k(S)-k(USU^\dagger)\in\mathbb Z$. We underline that this does not imply $\lim_{U\to I_N}\xi(USU^\dagger)=\xi(S)$ when $\exp[\xi_0(S)]<0$.
	
	We  want to show that the left hand side of Eq.~\eqref{diff_phi} is a continuous function in $U$, so that it can be only $n=0$ as $\U(N)$ is connected with the identity. For this purpose, it is sufficient to show that $\xi_0$ is continuous in $S$.
	It is not difficult to see that $\nu_\alpha(u)$ is continuous in $u$ and that $\tr RS$ is continuous in $S$. With $\tr R^2=1$, one has a bound for $|\nu_\alpha(\tr RS)|$ (c.f.~\cite[Lemma 8]{KZ23})
	\begin{equation}
		|\nu_\alpha(\tr RS)|\le
		\begin{cases}
			\displaystyle\frac{\|S\|^\alpha}{|\cos(\pi\alpha/2)|}, &\alpha\in(0,1)\cup(1,2],
			\vspace{1em}\\
			\displaystyle\|S\|\sqrt{1+\left(\frac{2\log\|S\|}{\pi}\right)^2}, &\alpha=1,
		\end{cases}
	\end{equation}
	where the bound on the right hand side is independent of $R$, which is integrable with respect to the normalised measure $H$. Therefore a dominated convergence theorem can be applied to obtain
	\begin{equation}
		 \lim_{\tr E^2\to 0}\int_{\mathbb S_{\Herm(N)}}H(\mathrm d R)\left[\nu_\alpha(\mathrm{Tr}\, R(S+E))-\nu_\alpha(\mathrm{Tr}\, RS)\right]=0.
	\end{equation}
	Together with the fact that $\tr Y_0S$ is continuous in $S$, we see that $\xi_0(S)$ is continuous in $S$.
	
	Since now the left hand side of~\eqref{diff_phi} is continuous in both $U$ and $S$, also $n=k(S)-k(USU^\dagger)$ is continuous in both $U$ and $S$, meaning it is a constant. By taking $U=I_N$ we see that $n=k(S)-k(USU^\dagger)\equiv 0$. Therefore for all $U\in\U(N)$ and $S\in\Herm(N)$, it is
	\begin{equation}
		\xi_0(USU^\dagger)=\xi_0(S).
	\end{equation}
	
	This prelude allows us to integrate over the $\U(N)$-action and exploit the group invariance,
	\begin{equation}
	\xi(S)=\int_{U\in\U(N)} \dv U\xi(USU^\dagger).
	\end{equation}
	Substituting~\eqref{3.1.3} into the right side, we arrive at
	\begin{equation}\label{3.1.12}
	\xi(S)=-\gamma\int_{\mathbb{S}_{\Herm(N)}} H(\dv R)\,\la\nu_\alpha(\tr USU^\dagger R)\ra+i\la\tr USU^\dagger Y_0\ra=:-\gamma I_1+I_2.
	\end{equation}
	The second term $I_2$ only depends on the eigenvalues of both $S$ and $Y_0$, and since the $U$-integral can be interchanged with the trace, one obtains
	\begin{equation}
		I_2=i\tr  \diag(s) \la U^\dagger Y_0U\ra=\frac{i}{N}\tr s\tr  Y_0.
	\end{equation}
	Setting $y_0=\tr  Y_0/N$ we notice that we could have used $Y_0=y_0I_N$ from the start.
	
	For the first term $I_1$, we denote $\tilde{H}(\dv R):=\la H(U\dv R U^\dagger)\ra$. Because $\nu_\alpha(\tr SR)$ is bounded for any $R$ on its domain and by the properties of the average over the Haar measure of $\U(N)$, one has
	\begin{equation}
		I_1=\int_{\mathbb{S}_{\Herm(N)}} \tilde{H}(\dv R)\nu_\alpha(\tr SR),\quad \tilde{H}(U\dv R U^\dagger)=\tilde{H}(\dv R),
	\end{equation}
	for all $U\in\U(N)$. Substituting these into~\eqref{3.1.12} and taking the exponentials on both sides show that $Y$ has the stable distribution $S(\alpha,\gamma \tilde{H},y_0I_N)$. Renaming $\tilde{H}\rightarrow H$ finishes this part of the proof.
	
	It is also easy to see (2)$\Rightarrow$(3) since the invariance of $S$ implies
	\begin{equation}\label{3.2.12}
		\mathbb E [\exp( i\tr  YS)]=\exp\left(-\gamma\int_{\mathbb{S}_{\Herm(N)}} H(\dv R)\nu_\alpha(\tr  USU^\dagger R)+i\tr  y_0 I_NS\right)
	\end{equation}
	for all $U\in\U(N)$. We choose a $U$ which diagonalises $S$ and notice that $\tr  \diag(s)R$ only depends on the diagonal entries of $R$. This means we can reduce the measure to the marginal distribution of the diagonal elements which is statement (3). 
	
	Next we are going to prove (2)$\Rightarrow$(4). This requires only an explicit calculation of the following equation~\eqref{3.2.14}, which is given by taking an average over $U$ in~\eqref{3.2.12}
	\begin{equation}\label{3.2.14}
		\mathbb E [\exp( i\tr  YS)]=\exp\left(-\gamma\int_{\mathbb{S}_{\Herm(N)}} H(\dv R)\la\nu_\alpha(\tr  U\diag(s)U^\dagger \diag(r))\ra+iy_0 1_N^\top s\right),
	\end{equation}
	where $s$ and $r$ are respectively the eigenvalues of $S$ and $R$. We now simplify $\la\nu_\alpha(\tr  U\diag(s)U^\dagger \diag(r))\ra$ using Lemma~\ref{p3.1.1}. In the $\alpha\ne 1$ case one substitutes~\eqref{3.1.5} and~\eqref{3.1.6} into~\eqref{3.1.12} to see
	\begin{equation}\label{3.2.15}
	\begin{split}
		\la\nu_\alpha(\tr  USU^\dagger R)\ra=
		&\frac{\Gamma(\alpha+1)}{2\pi i}\Bigg(\int_{1+i\R}\frac{\dv z}{z^{\alpha+1}} \Big\langle\exp[z\tr  U\diag(s)U^\dagger \diag(r)]+\exp[-z\tr  U\diag(s)U^\dagger \diag(r)]\Big\rangle\\
		&\hspace*{-0.5cm}-i\tan\frac{\pi\alpha}{2}\int_{1+i\R}\frac{\dv z}{z^{\alpha+1}} \Big\langle\exp[z\tr  U\diag(s)U^\dagger \diag(r)]-\exp[-z\tr  U\diag(s)U^\dagger \diag(r)]\Big\rangle \Bigg).
	\end{split}
	\end{equation}
 	Substituting the HCIZ integral~\eqref{2.2.5} into this equation and afterwards expanding the determinant give
	\begin{equation}
	\begin{split}
		\la\nu_\alpha(\tr  U\diag(s)U^\dagger \diag(r))\ra=
		&\frac{\prod_{j=0}^{N-1}j!}{\Delta(s)\Delta(r)} \sum_{\rho\in\mathrm{S}_N}\sgn(\rho)\frac{\Gamma(\alpha+1)}{2\pi i}\\
		&\times\Bigg(\int_{1+i\R}\frac{\dv z}{z^{\alpha+1+N(N-1)/2}} \left(\prod_{j=1}^N e^{zs_jr_\rho(j)}+(-1)^{N(N-1)/2} \prod_{j=1}^N e^{-zs_jr_\rho(j)}\right)\\
		&\hspace*{-1.5cm}-i\tan\left(\frac{\pi\alpha}{2}\right)\int_{1+i\R}\frac{\dv z}{z^{\alpha+1+N(N-1)/2}} \left(\prod_{j=1}^N e^{zs_jr_\rho(j)}-(-1)^{N(N-1)/2} \prod_{j=1}^N e^{-zs_jr_\rho(j)}\right)\Bigg).
	\end{split}
	\end{equation}
	We can now apply Lemma~\ref{p3.1.1} again to simplify the average to
	\begin{equation}
		\la\nu_\alpha(\tr  USU^\dagger R)\ra=\frac{\prod_{j=0}^{N-1}j!}{\Delta(s)}\frac{\Gamma(\alpha+1)}{\Gamma(\alpha+1+N(N-1)/2)}\sum_{\rho\in\mathrm S_N}\frac{\nu_{\alpha}(s^\top r_\rho)(\sgn(s^\top r_\rho)|s^\top r_\rho|)^{N(N-1)/2}}{\Delta(r_\rho)}.
	\end{equation}
	Notice $\sgn(s^\top r_\rho)|s^\top r_\rho|=s^\top r_\rho$. To modify this further to the form as~\eqref{3.1.10}, one notices that the $H$ integral after normalisation is only an expectation of eigenvalues of $H$. So one can replace the measure $H$ with its eigenvalue part $h_\eig$.
	
	For $\alpha=1$, we exploit the relation $x\log |x|=\partial_a\big|_{a=1}|x|^a$ for $x\ne 0$. Again by making use of Lemma~\ref{3.1.8}, one obtains
	\begin{equation}
	\begin{split}
		\la\nu_1(\tr  USU^\dagger R)\ra=
		&\frac{1}{2\pi i}\bigg(\int_{1+i\R}\frac{\dv z}{z^{2}} \Big\langle\exp[z\tr U\diag(s)U^\dagger \diag(h)]+\exp[-z\tr [U\diag(s)U^\dagger \diag(h)]\Big\rangle\\
		&+\frac{2i}{\pi}\frac{\partial}{\partial a}\bigg|_{a=1}\int_{1+i\R}\frac{\dv z}{z^{a+1}} \Big\langle\exp[z\tr U\diag(s)U^\dagger \diag(h)]+\exp[-z\tr U\diag(s)U^\dagger \diag(h)]\Big\rangle\bigg).
	\end{split}	
	\end{equation}
	The interchange of the $U$-integral and the derivative can be justified by Lebesgue's dominated convergence theorem as the integrand is uniformly bounded. Then one can obtain the result of the $\alpha=1$ case by the same reasoning as for $\alpha\neq1$.
	
	Finally, as we have both (2)$\Rightarrow$(3) and (2)$\Rightarrow$(4), their converses are easily shown by the uniqueness of the derivative principle, i.e. the one-to-one relation between the invariant matrix distribution, eigenvalue distribution and diagonal entry distribution.
\end{proof}

The proof of Theorem~\ref{p3.1a} is rather straightforward, as it is evaluating the right hand side of~\eqref{3.1.8a}.

\begin{proof}[Proof of Theorem~\ref{p3.1a}]
	For $\alpha\ne 1$, we are going to show
	\begin{equation}
		\int_{\|t\|\le 1}h_\diag(\dv t)\nu_\alpha(s^\top t)=\int_{r\in\mathbb S^{N-1}}g_\alpha(\dv r)\nu_\alpha(s^\top r).
	\end{equation}
	This is obtained by noticing $\nu_\alpha(s^\top t)=\|t\|^\alpha\nu_\alpha(s^\top t/\|t\|)$, and for each $s$ the function $\nu_\alpha(s^\top t)$ (in terms of $t$) is bounded inside the unit ball.
	
	For $\alpha=1$, we instead have
	\begin{equation}
		\int_{\|t\|\le 1}h_\diag(\dv t)\nu_1(s^\top t)=\int_{r\in\mathbb S^{N-1}}g_1(\dv r)\nu_1(s^\top r)+\int_{\|t\|\le 1}h_\diag(\dv t)\frac{2i}{\pi}\, s^\top t\log \|t\|.
	\end{equation}
	This second term can be absorbed into the shift as $iy_11_N^\top s$ as the permutation invariance induced by the unitary invariance of the measure $H$ tells us that the average over $t\log \|t\|$ with respect to the measure $h_\diag(\dv t)$ must be proportional to $1_N$.
	
	The converse, meaning the statements of Theorem~\ref{p3.1a} imply the third statement in Theorem~\ref{p3.1}, is simply obtained by the uniqueness of the spectral measure.
\end{proof}

\section{Examples}\label{s5}

We will construct several stable invariant ensembles by choosing particular spectral measures. Theorem~\ref{p3.1} part (2) allows us to set the shift to be $y_0I_N$ for any $y_0\in\R$, as well as the spectral measure following $H(\dv R)=H(U\dv R U^\dagger )$ for any $U\in\U(N)$. Without loss of generality one always assumes there is no shift.

\subsection{Elliptical Stable Ensembles}\label{sec.ell}

In the previous section, we have discussed rather general stable and invariant ensembles. To illustrate our framework we would like to ask what the corresponding elliptical stable invariant ensembles are. 

\begin{definition}[Elliptical Stable Matrix Ensemble]\label{def:elliptic}\ 

An elliptical stable random matrix $Y\in\Herm(N)$ has the characteristic function
\begin{equation}\label{elliptic}
\mathbb E [\exp( i\tr  YS)]=\exp \left[-\left(\sum_{a,b,c,d=1}^N\Sigma_{ab,cd}S_{ab}S_{cd}^*\right)^{\alpha/2}+i\tr  Y_0 S\right]
\end{equation}
with $S_{cd}^*$ the complex conjugate of $S_{cd}$, $0<\alpha\leq 2$,  a fixed matrix $Y_0\in\Herm(N)$ and $\Sigma$ is a $N^2\times N^2$ Hermitian positive definite matrix satisfying the condition $H_{ab,cd}=H_{ba,dc}^*$.
\end{definition}

This definition is exactly the definition of elliptical ensembles for real vectors after we have identified $\Herm(N)$ with $\R^{N^2}$. The condition $H_{ab,cd}=H_{ba,dc}^*$ only reflects that $\{\sum_{a,b,c,d=1}^N\Sigma_{ab,cd}S_{ab}\}_{c,d=1,\ldots, N}$ is still a Hermitian matrix. Thus, $\Sigma$ is actually a positive real symmetric matrix written in uncommon basis, because $\Herm(N)$ is a real vector space. 

These ensembles can be readily created via a Gaussian ensemble and the randomly distributed variance. The construction follows the one used in~\cite{BCP2008,AV2008,AAV2009,AFV2010} where general elliptical matrix ensembles were created that have been, however, not only stable ones.

\begin{proposition}[Relation Elliptical Stable Ensembles and their Gaussian Counterparts]\label{prop:Gauss.ellip}\

Let $Y\in\Herm(N)$ be a random matrix drawn from the elliptical ensemble of Definition~\ref{def:elliptic} with $\alpha\in(0,2)$. Then it can be generated by a Gaussian elliptical ensemble with the density
\begin{equation}\label{elliptic.Gaussian}
P(Y|\sigma_0,Y_0,\Sigma)=\frac{1}{2^{N(N-1)/2}\pi^{N^2/2}\sigma_0^{N^2}\sqrt{\det\Sigma}}\exp \left[-\frac{1}{2\sigma_0^2}\sum_{a,b,c,d=1}^N\{\Sigma^{-1}\}_{ab,cd}(Y-Y_0)_{ab}(Y-Y_0)_{cd}^*\right],
\end{equation}
where $\sigma_0^2/2>0$ is a random variable drawn from the L\'evy $(\alpha/2)$-stable distribution with the characteristic function
\begin{equation}\label{uniform.Levy}
\mathbb E \left[\exp\left(-\frac{\sigma_0^2 k}{2}\right)\right]=\exp[-k^{\alpha/2}]\ {\rm for\ any}\ k\geq0.
\end{equation}
The corresponding univariate probability density will be denoted by $p_\alpha$.
\end{proposition}

From the point of view of this proposition, it becomes immediate that the corresponding random matrix ensemble is stable.

\begin{proof}
We prove this statement with the help of the uniqueness of the characteristic function. Here we have to take into account the fact that we need to average over both $Y$ and $\sigma_0$. The average over $Y$ is trivial as it is a multivariate Gaussian integral, i.e.
\begin{equation}
\mathbb{E}[e^{i\tr  YS}]=\mathbb E \left[\exp \left(-\frac{\sigma_0^2}{2}\sum_{a,b,c,d=1}^N\Sigma_{ab,cd}S_{ab}S_{cd}^*+i\tr  Y_0 S\right)\right].
\end{equation}
On the right hand side we still need to integrate over $\sigma_0^2$. In the final step we make use of~\eqref{uniform.Levy} to arrive at~\eqref{elliptic}.
\end{proof}

As mentioned before we are interested in matrix ensembles that are invariant under the conjugate group action of $\U(N)$. Then the matrix $\Sigma$ simplifies drastically.

\begin{lemma}[Elliptical Stable Invariant Matrix Ensemble]\label{p6.1}\

	Assuming the notation of Definition~\ref{def:elliptic} and that $Y$ is an invariant ensemble. Then, there are two constants $\sigma>0$ and $\kappa>-\sigma^2/N$ such that it holds
	\begin{equation}\label{Sigma}
	\Sigma_{ab,cd}=\frac{\sigma^2}{\alpha} \delta_{ac}\delta_{bd}+\frac{\kappa}{\alpha} \delta_{ab}\delta_{cd}
	\end{equation}
	or equivalently $\sum_{a,b,c,d=1}^N\Sigma_{ab,cd}S_{ab}S_{cd}^*=\sigma^2\tr S^2/\alpha+\kappa(\tr S)^2/\alpha$ and $Y_0=y_0I_N$.
\end{lemma}

This lemma follows essentially from the representation theory of $\U(N)$. The condition $\kappa>-\sigma^2/N$ guarantees the integrability of the density, especially the integrability of the random variable $\tr S$. 

\begin{proof}
We make use of the representation in terms of the elliptical Gaussian random matrix ensemble as the average over $\sigma_0$ does not introduce any symmetry breaking of the group invariance. Hence, the invariance of the elliptical Gaussian under the conjugate action of $\U(N)$ tells us that the exponent of the right hand side of~\eqref{elliptic.Gaussian} has to be the same for $Y$ and any $UYU^\dagger$ with $U\in\U(N)$. As the exponent is a polynomial of order $2$, we can ask for the possible invariance under this conjugate action. Those are given by the Cayley-Hamilton theorem and are given by a constant, the linear function $c_0\tr  Y$ and the two quadratic invariants $c_1\tr  Y^2+c_2(\tr  Y)^2$. This is equivalent with the statement of the lemma.
\end{proof}

The relation to the previous section can be made once we have identified the corresponding spectral measure $H(\dv R)$. We prove this in the next proposition. This statement is a drastic simplification of the spectral measure for general elliptical stable ensembles in~\cite[Proposition 2.5.8]{ST94}.

\begin{proposition}[Spectral Measure of Elliptical Stable Invariant Matrix Ensembles]\label{prop:spec.ellipt}\

Let $N>1$ and we assume the setting of Lemma~\eqref{p6.1}. Then the elliptical ensemble of Lemma~\eqref{p6.1} is drawn from the stable ensemble $S(\alpha,\gamma H, y_0 I_N)$ with the spectral measure and scaling
\begin{equation}\label{ellip.spec}
H(\dv R)=c_3\frac{\mu(\dv R)}{[c_1+c_2(\tr  R)^2]^{(\alpha+N^2)/2}}\quad {\rm and}\quad \gamma=\frac{\Gamma[N^2/2]\Gamma[\alpha+1]}{c_1^{(N^2-1)/2}c_3\sqrt{c_1+Nc_2}\Gamma[(\alpha+N^2)/2]\Gamma[\alpha/2+1]}.
\end{equation}
The constants are
\begin{equation}\label{param.ell}
c_1=\frac{\alpha}{4\sigma^2},\ c_2=-\frac{\alpha \kappa}{4\sigma^2(\sigma^2+N\kappa)}>-\frac{c_1}{N},\ {\rm and}\   c_3=\frac{c_1^{(\alpha+N^2)/2}}{{\ _2F_1}\left(\left.\frac{1}{2},\frac{\alpha+N^2}{2};\frac{N^2}{2}\right|-\frac{c_2 N^2}{c_1}\right)}.
\end{equation}
The measure $\mu(\dv R)$ is the normalised uniform measure on the hypersphere $\mathbb{S}_{\Herm(N)}$ and ${\ _2F_1}$ is the hypergeometric function.
\end{proposition}

\begin{proof}
We need to show that
\begin{equation}
\gamma \int_{\mathbb{S}_{{\rm Herm}(N)}}H({\rm d}R)|\tr SR|^\alpha\left(1-i \sgn(\tr SR)\tan\frac{\pi \alpha}{2}\right)=\left[\frac{\sigma^2}{\alpha}\tr S^2+\frac{\kappa}{\alpha}(\tr S)^2\right]^{\alpha/2}.
\end{equation}
The integral over $\sgn(\tr SR)|\tr SR|^\alpha$ vanishes as the measure is invariant under the reflection $R\to-R$. For this purpose we choose two constants $c_1>0$ and $c_2>-c_1 N$ and introduce the two integrals
\begin{equation}
\begin{split}
I_1=\int_{\mathbb{S}_{\Herm(N)}}\frac{\mu(\dv R)}{[c_1+c_2(\tr  R)^2]^{(\alpha+N^2)/2}},\quad I_2=\int_{\mathbb{S}_{\Herm(N)}}\frac{|\tr  S R|^\alpha \mu(\dv R)}{[c_1+c_2(\tr  R)^2]^{(\alpha+N^2)/2}}
\end{split}
\end{equation}
with $\mu$ the uniform normalised measure on the sphere $\mathbb{S}_{\Herm(N)}$. The second condition $c_2>-c_1 N$ results from the fact that $(\tr  R)^2\leq N$ when $\tr  R^2=1$. The maximal value is, for example, achieved by $R=I_N/\sqrt{N}$. The integral $I_1=1/c_3$ is the normalisation of the spectral measure $H({\rm d}R)$ while $I_2$ is essentially the function we need to find, i.e,
\begin{equation}
 \int_{\mathbb{S}_{{\rm Herm}(N)}}H({\rm d}R)|\tr SR|^\alpha=\frac{I_2}{I_1}.
\end{equation}

The first integral $I_1$ is needed to find the proper normalisation and $I_2$ shows that we indeed get the elliptical stable invariant ensemble. The signum part in $\nu_\alpha$ can be omitted as it is odd under the reflection $R\rightarrow-R$ while the measure is even. The parameters $c_1$, $c_2$ and $\gamma$ will be fixed at the end by identification with~\eqref{elliptic} and~\eqref{Sigma}.

In the first integral we split the matrix $R=r_1I_N/\sqrt{N}+\tilde{R}$ into $r_1\in[-1,1]$ and its traceless part $\tilde{R}$, i.e., $\tr \tilde{R}=0$. When understanding $R$ is an $N^2$ dimensional normalised real vector and $I_N/\sqrt{N}$ as one specific direction parametrised by the component $r_1$. We can integrate out all other components. This leads to the normalised marginal measure of $r_1$ which is given by 
\begin{equation}
\frac{\Gamma[N^2/2]}{\sqrt{\pi}\Gamma[(N^2-1)/2]}(1-r_1^2)^{(N^2-3)/2} \dv r_1,
\end{equation}
see also~\cite[Eq.~(3.113)]{Fo10}. Then, the integral is equal to the hypergeometric function,
\begin{equation}
I_1=\frac{\Gamma[N^2/2]}{\sqrt{\pi}\Gamma[(N^2-1)/2]}\int_{-1}^1\dv r_1 \frac{(1-r_1^2)^{(N^2-3)/2}}{[c_1+c_2N^2r_1^2]^{(\alpha+N^2)/2}}=\frac{1}{c_1^{(\alpha+N^2)/2}} {\ _2F_1}\left(\left.\frac{1}{2},\frac{\alpha+N^2}{2};\frac{N^2}{2}\right|-\frac{c_2 N^2}{c_1}\right).
\end{equation}

For the second integral we make use of the identity
\begin{equation}
\begin{split}
\frac{1}{[c_1+c_2(\tr  R)^2]^{(\alpha+N^2)/2}}=&\frac{2}{\Gamma[(\alpha+N^2)/2]}\int_0^\infty \frac{\dv r}{r}\, r^{\alpha+N^2} \exp\left[-(c_1+c_2(\tr  R)^2)r^2\right]\\
=&\frac{2}{\Gamma[(\alpha+N^2)/2]}\int_0^\infty \frac{\dv r}{r}\int_{-\infty}^\infty\frac{\dv t}{\sqrt{\pi}} r^{\alpha+N^2} \exp\left[-c_1 r^2-t^2+2i\sqrt{c_2 }r\,\tr  R\,t\right].
\end{split}
\end{equation}
We define $V=rR$ where $r$ plays the role of the norm of $V$, i.e., $r=\sqrt{\tr  V^2}$. Then, the integration over $r$ and $R$ can be combined to the integration over the whole Hermitian matrices with $r^{N^2-1}\dv r\mu(\dv R)=c\dv V$ the Lebesgue measure on $\Herm(N)$ (in particular the product of the differentials of all real components of $V$)  times a normalisation constant $c$. This constant can be fixed via
\begin{equation}
\begin{split}
\frac{\Gamma[N^2/2]}{2}=&\int_{\mathbb{S}_{\Herm(N)}}\mu(dR)\int_0^\infty \frac{\dv r}{r} r^{N^2}\exp[-r^2]=c\int_{\Herm(N)}\dv V\exp[-\tr V^2]=\frac{\pi^{N^2/2}}{2^{N(N-1)/2}}c
\end{split}
\end{equation}
because the normalisation of the measure is independent of the integrand. Thence, the integral $I_2$ is equal to
\begin{equation}
\begin{split}
 I_2=&\frac{2^{N(N-1)/2}\Gamma[N^2/2]}{\pi^{N^2/2}\Gamma[(\alpha+N^2)/2]}\int_{\Herm(N)}\dv V\int_{-\infty}^\infty\frac{\dv t}{\sqrt{\pi}}  |\tr SV|^{\alpha}\exp\left[-c_1 \tr V^2-t^2+2i\sqrt{c_2 }\,\tr  V\,t\right].
\end{split}
\end{equation}
This integral can be evaluated with the help of~\eqref{3.1.5}.

We would like to highlight that all three integrations over $V$, $t$ and $z$ are absolutely convergent so that we can employ Fubini's theorem to interchange those integrals. The computation is as follows
\begin{equation}
\begin{split}
 I_2=&\frac{2^{N(N-1)/2}\Gamma[N^2/2]\Gamma[\alpha+1]}{\pi^{N^2/2}\Gamma[(\alpha+N^2)/2]}\int_{\Herm(N)}\dv V\int_{-\infty}^\infty\frac{\dv t}{\sqrt{\pi}}  \int_{1+i\R}\frac{\dv z}{2\pi i z^{\alpha+1}}\left(e^{z\tr  SV }+e^{-z\tr  SV }\right)\\
 &\times\exp\left[-c_1 \tr V^2-t^2+2i\sqrt{c_2 }\,\tr  V\,t\right]\\
 =&\frac{\Gamma[N^2/2]\Gamma[\alpha+1]}{c_1^{N^2/2}\Gamma[(\alpha+N^2)/2]}\int_{-\infty}^\infty\frac{\dv t}{\sqrt{\pi}}  \int_{1+i\R}\frac{\dv z}{2\pi i z^{\alpha+1}}e^{-t^2}\\
 &\times\left(\exp\left[-\frac{1}{c_1}\tr\left(\sqrt{c_2} tI_N-\frac{i}{2}zS\right)^2\right]+\exp\left[-\frac{1}{c_1}\tr\left(\sqrt{c_2} tI_N+\frac{i}{2}zS\right)^2\right]\right)\\
 =&\frac{\Gamma[N^2/2]\Gamma[\alpha+1]}{c_1^{(N^2-1)/2}\sqrt{c_1+Nc_2}\Gamma[(\alpha+N^2)/2]} \int_{1+i\R}\frac{\dv z}{\pi i z^{\alpha+1}}\exp\left[-\frac{c_2}{4(c_1+c_2 N)c_1}(\tr  S)^2z^2+\frac{\tr S^2}{4c_1}z^2\right].
\end{split}
\end{equation}
One can indeed readily check that the pre-factor of $z^2$ is positive because of the condition $c_2N>-c_1$ as it is
\begin{equation}
-\frac{c_2(\tr  S)^2}{4(c_1+c_2 N)c_1}+\frac{\tr S^2}{4c_1}=\frac{(\tr  S)^2}{4(c_1+c_2 N)}+\frac{\tr (S-[\tr  S/N]I_N)^2}{4c_1}>0,\quad{\rm for}\ S\neq0.
\end{equation}
Moreover, $z^2$ describes a contour which encloses the negative real axis. Therefore, we can substitute $\tilde{z}=z^2$ and deform the contour of $\tilde{z}$ again back to $1+i\R$  so that we can apply~\eqref{3.1.5}, anew, only that now the exponent is $\alpha/2$ instead of $\alpha$ due to the rooting. Therefore, we arrive at
\begin{equation}
\begin{split}
 I_2=&\frac{\Gamma[N^2/2]\Gamma[\alpha+1]}{c_1^{(N^2-1)/2}\sqrt{c_1+Nc_2}\Gamma[(\alpha+N^2)/2]\Gamma[\alpha/2+1]} \left[-\frac{c_2}{4(c_1+c_2 N)c_1}(\tr  S)^2+\frac{\tr S^2}{4c_1}\right]^{\alpha/2}.
\end{split}
\end{equation}
The identification with~\eqref{Sigma}, in particular
\begin{equation}
\gamma\frac{I_2}{I_1}=\left(\frac{\sigma^2}{\alpha}\tr  S^2+\frac{\kappa}{\alpha}(\tr S)^2\right)^{\alpha/2},
\end{equation}
yields the scaling.
This closes the proof.
\end{proof}

As we know now the spectral measure on the matrix level we can ask for the marginal measures for the diagonal entries and the eigenvalues of the corresponding random matrix.

\begin{corollary}[Marginal Spectral Measures of Elliptical Stable Invariant Matrix Ensembles]\label{cor:}\

Considering the setting of Proposition~\ref{prop:spec.ellipt}, the marginal spectral measure for the diagonal entries is given by
\begin{equation}\label{ellip.spec.diag}
h_{\rm diag}(\dv t)=\frac{\Gamma[N^2/2]c_3}{\pi^{N/2}\Gamma[N(N-1)/2]}\frac{(1-\|t\|^2)^{N(N-1)/2-1}\dv t}{[c_1+c_2(t^\top 1_N)^2]^{(\alpha+N^2)/2}}
\end{equation}
and the one for the eigenvalues is
\begin{equation}\label{ellip.spec.eig}
h_{\rm eig}(\dv r)=\frac{\Gamma[N^2/2]c_3}{\Gamma[N/2]\prod_{j=0}^Nj!}\frac{\Delta^2(r)\tilde{\mu}(\dv r)}{[c_1+c_2(r^\top 1_N)^2]^{(\alpha+N^2)/2}}
\end{equation}
with $\tilde{\mu}$ the normalised uniform measure on the hypersphere $\mathbb{S}^{N-1}$ and $\dv t=\prod_{j=1}^N \dv t_j$.
\end{corollary}

\begin{proof}
The two formulas follow from $\tr  R=t^\top 1_N=r^\top 1_N$ when $t$ contains the diagonal entries of $R$ and $r$ is its eigenvalues in form of a column vector. Thus, we only need to evaluate what the normalised uniform measure
\begin{equation}
\mu(\dv R)=\frac{\delta(1-\tr R^2)\dv R}{\int_{\Herm(N)}\delta(1-\tr R^2)\dv R}
\end{equation}
looks like. The denominator properly normalises the measure and $\dv R$ is the product of all $N^2$ differential of the matrix elements of $R\in\Herm(N)$.

For the diagonal entries we split the matrix  $R=\diag(t)+\tilde R$ with $\tilde{R}$ having only zero diagonal entries. Hence it is $\tr R^2=\|t\|^2+\tr \tilde{R}^2$. Rescaling $\tilde{R}\to\sqrt{1-\|t\|^2}\tilde{R}$ factorises the measure into a normalised one for $\tilde{R}$ and the one for $t$ which is
\begin{equation}
\begin{split}
\frac{(1-\|t\|^2)^{N(N-1)/2-1}\dv t}{\int_{\|t\|^2\leq 1}(1-\|t\|^2)^{N(N-1)/2-1}\dv t}=&\frac{\Gamma[N/2](1-\|t\|^2)^{N(N-1)/2-1}\dv t}{2\pi^{N/2}\int_{0}^1(1-s^2)^{N(N-1)/2-1}s^{N-1}\dv s}\\
=&\frac{\Gamma[N^2/2]}{\pi^{N/2}\Gamma[N(N-1)/2]}(1-\|t\|^2)^{N(N-1)/2-1}\dv t.
\end{split}
\end{equation}
In the first equality, we have used the polar decomposition with the radius $s=\|t\|$, and we have employed the volume of an $(N-1)$-dimensional sphere.  This proves~\eqref{ellip.spec.diag}.

For the eigenvalues, we diagonalise the matrix $R=U\diag(r)U^\dagger$  while the integration of the eigenvectors factorises. The measure becomes
\begin{equation}
\mu(\dv R)=\frac{\delta(1-\tr r^2)\Delta^2(r)\dv r\dv U}{\int_{\R^N}\delta(1-\tr r^2)\Delta^2(r)\dv r}
\end{equation}
with $\dv U$ the normalised Haar measure of $\U(N)/\U^N(1)$ and $\Delta(r)$ the Vandermonde determinant. The normalised uniform measure on $\mathbb{S}^{N-1}$ is equal to
\begin{equation}
\tilde{\mu}(\dv r)=\frac{\delta(1-\tr r^2)\dv r}{\int_{\R^N}\delta(1-\tr r^2)\dv r}.
\end{equation}
Hence, we need to compute the ratio
\begin{equation}
\begin{split}
\frac{\int_{\R^N}\delta(1-\tr r^2)\dv r}{\int_{\R^N}\delta(1-\tr r^2)\Delta^2(r)\dv r}=&\frac{\Gamma[N^2/2]\int_0^\infty s^{N-1}e^{-s^2/2}ds\int_{\R^N}\delta(1-\tr r^2)\dv r}{\Gamma[N/2]\int_0^\infty s^{N^2-1}e^{-s^2/2}ds\int_{\R^N}\delta(1-\tr r^2)\Delta^2(r)\dv r}\\
=&\frac{\Gamma[N^2/2]\int_{\R^N}e^{-\|u\|^2/2}\dv u}{\Gamma[N/2]\int_{\R^N}e^{-\|u\|^2/2}\Delta^2(u)\dv u}.
\end{split}
\end{equation}
Here, we have introduced a radial component for the normalised vector $r$ such that $u=s r$ is the polar decomposition applied in reverse. The numerator is a pure Gaussian integral while the denominator is a well-known Selberg integral~\cite[Chapter 4.3.2]{Fo10}. Such that the ratio is equal to
\begin{equation}
\begin{split}
\frac{\int_{\R^N}\delta(1-\tr r^2)\dv r}{\int_{\R^N}\delta(1-\tr r^2)\Delta^2(r)\dv r}=&\frac{\Gamma[N^2/2]}{\Gamma[N/2]\prod_{j=0}^Nj!}.
\end{split}
\end{equation}
This yields the proper normalisation of~\eqref{ellip.spec.eig}.
\end{proof}

The simplest elliptical stable invariant matrix ensemble is the one with a uniform spectral measure, in particular $H(\dv R)$ is the uniform normalised measure on the sphere $\mathbb{S}_{\Herm(N)}$. Employing our result~\eqref{ellip.spec} in Proposition~\ref{prop:spec.ellipt} with $\kappa=0$, we see that it is given by $c_2=0$. The corresponding characteristic function is
\begin{equation}\label{6.1.6}
	\E [\exp(i\tr YS)]=\exp\left[-\left(\frac{\sigma^2}{2}\tr S^2\right)^{\alpha/2}\right].
\end{equation}
Recalling the notation $p_\alpha$ for the L\'evy $\alpha$-stable distribution, see~\eqref{uniform.Levy}, we can write the corresponding random matrix density as follows
\begin{equation}
F(Y)=\frac{1}{2^{N(N-1)/2}(\pi\sigma)^{N^2/2}}\int_0^\infty \frac{\dv x}{x^{N^2/2}} p_\alpha(x)\exp\left[-\frac{\tr Y^2}{2x\sigma^2}\right].
\end{equation}
This kind of ensemble is even strictly stable as can be shown either by direct computation of the convolution of two statistically independent copies of the corresponding random matrix or by simply showing~\eqref{3.1.10a}. Similar ensembles given as averages over the Gaussian unitary ensemble have been studied before~\cite{BCP2008,AFV2010,AV2008,AAV2009}.

Finally we are going to look at a special case where $\alpha=2$ which is the Gaussian case. From probability theory we know that an $\R^d$ stable vector with $\alpha=2$ must follow a multivariate normal distribution. Hence, it is elliptical from the start. Moreover, its marginal distribution (the diagonal entries) must also be a multivariate normal distribution. From a classical random matrix point of view these matrices can be generated by $X=G+a\tr (G)\ I_N$ where $G$ is drawn from a Gaussian Unitary Ensemble and $a\in\mathbb{R}$ is a constant. We summarise the following classification of the $\alpha=2$ case.

\begin{corollary}[$\alpha=2$]\label{alpha=2}\
	
	We consider the setting of Proposition~\ref{prop:spec.ellipt} for $\alpha=2$. Then, every invariant stable ensemble is elliptical with the joint probability density of the eigenvalues
	\begin{equation}\label{alpha=2a}
	\begin{split}
	f_{\rm eig}(x)=&\frac{1}{\prod_{j=0}^Nj!}\Delta(x)\Delta(-\partial_x)\frac{1}{(2\pi)^{N/2}\sigma^{N-1}\sqrt{\sigma^2+N\kappa}}\exp\left(-\frac{\|x\|^2}{2\sigma^2}+\frac{\kappa(x^\top 1_N)^2}{2(\sigma^2+N\kappa)\sigma^2}\right)\\
	=&\frac{1}{(2\pi)^{N/2}\sigma^{N^2-1}\sqrt{\sigma^2+N\kappa}\prod_{j=0}^Nj!}\Delta(x)^2\exp\left(-\frac{\|x\|^2}{2\sigma^2}+\frac{\kappa(x^\top 1_N)^2}{2(\sigma^2+N\kappa)\sigma^2}\right).
	\end{split}
	\end{equation}
\end{corollary}
\begin{proof}
 We first compute the joint probability density of the diagonal entries $x$ which is given by the inverse Fourier transform
 \begin{equation}
f_{\diag}(x)= \int_{\R^N}\frac{\dv t}{(2\pi)^N}\exp\left[-\frac{\sigma^2}{2}\|t\|^2-\frac{\kappa}{2}(t^\top 1_N)^2-i x^\top t\right].
 \end{equation}
This integral can be computed readily once we shift $t\rightarrow t+ix^\top1_N/[\sigma^2+N\kappa]$. It yields
 \begin{equation}
f_{\diag}(x)=\frac{1}{(2\pi)^{N/2}\sigma^{N-1}\sqrt{\sigma^2+N\kappa}}\exp\left(-\frac{\|x\|^2}{2\sigma^2}+\frac{\kappa(x^\top 1_N)^2}{2(\sigma^2+N\kappa)\sigma^2}\right).
 \end{equation}
	Then, we can apply the derivative principle~\eqref{2.3.2} to find the first line of~\eqref{alpha=2a}. The second line follows from the fact that any differentiable function of $x^\top 1_N$ commutes with $\Delta(-\partial_x)$, and the action of the latter on  the remaining exponential function yields a determinant of Hermite polynomials from order $0$ to order $N-1$, i.e.,
	\begin{equation}
	\begin{split}
	e^{\|x\|^2/(2\sigma^2)}\Delta(-\partial_x)e^{-\|x\|^2/(2\sigma^2)}=&\det\left[e^{x_a^2/(2\sigma^2)}\partial_{x_a}^{b-1}e^{-x_a^2/(2\sigma^2)}\right]_{a,b=1,\ldots,N}=\frac{(-1)^{N(N-1)/2}}{\sigma^{N(N-1)}}\det[x_a^{b-1}]_{a,b=1,\ldots,N},
	\end{split}
	\end{equation}
	where we have exploited the properties of the determinant. This is the reverse direction as the orthogonal polynomial method in random matrix theory works.
\end{proof}

\subsection{Dirac Point Mass Spectral Measures}\label{sec:Dirac}

Another simple choice of the spectral measure is a Dirac point mass measure. Especially we consider
\begin{equation}\label{5.1.1}
	\int_{\mathbb S_{\Herm(N)}} \varphi(R)H_{\mathrm{Dirac}}(\dv R):=\sum_{m=1}^\infty p_m\varphi(R_m)
\end{equation}
for any bounded measurable function $\varphi\in B(\mathbb S_{\Herm(N)})$. Here $R_m$ is a point mass on the unit sphere with weight $p_m\ge 0$, satisfying the normalisation constraint
\begin{equation}
	\sum_{m=1}^\infty p_m=1.
\end{equation}
The question is now which kind of point mass measure is also invariant under the conjugate action of $\U(N)$. From the point of view of representation theory only a multiple of the identity matrix commutes with an arbitrary unitary matrix. However, on the level of measures we need to make sure that the measure cannot absorb the matrix $U$. For that one needs to show that the orbit of a matrix $X\in\Herm(N)$ defined by
\begin{equation}
	\mathcal O_{X}:=\left\{UXU^\dagger| U\in\U(N) \right\},
\end{equation}
is always at least one-dimensional when $X$ is not proportional to the identity matrix. Thus, the invariant spectral measure cannot be a Dirac point mass measure and $\pm I_N/\sqrt{N}$ are the only two eligible matrices for $R_M$.

\begin{proposition}[Invariant Dirac Point Mass Spectral Measures]\label{discrete_spectral_measure}\

	If $H_{\mathrm{Dirac}}$ is an invariant spectral measure containing only Dirac point masses, i.e., it holds~\eqref{5.1.1}. Then, there is a $p\in[0,1]$ such that
	\begin{equation}\label{point.mass.spectral}
		\int_{\mathbb S_{\Herm(N)}} \varphi(R)H_{\mathrm{Dirac}}(\dv R)=p\varphi(I_N/\sqrt{N})+(1-p)\varphi(-I_N/\sqrt{N})
	\end{equation}
	for any bounded measurable function $\varphi\in B(\mathbb S_{\Herm(N)}))$.
\end{proposition}

\begin{proof}
	We will firstly show a preliminary result. For any $X\in\Herm(N), X\ne kI_N$ and $\varepsilon>0$, there exists a $U\in\U(N)$ so that
	\begin{equation}\label{6.1.6a}
		0<\tr(UXU^\dagger-X)^2<\varepsilon.
	\end{equation} 
	In other words, we find a $U\in\U(N)$ such that $UXU^\dagger\neq X$ is infinitesimally close to $X$. To show this assertion, we choose a $H\in\Herm(N)$ such that $HX-XH\ne 0$, which is always possible because of $X\ne kI_N$. Then for some small $\delta>0$, we consider the unitary matrix as $U=\exp( i\delta H)$. As the exponential map is continuous and differentiable we have
	\begin{equation}
		\lim_{\delta\rightarrow 0}\tr(UXU^\dagger-X)^2=0\quad {\rm while}\quad \lim_{\delta\rightarrow 0}\frac{\dv}{\dv\delta}\left(UXU^\dagger-X\right)=HX-XH\ne 0.
	\end{equation}
	Hence, $UXU^\dagger\neq X $ for suitably small $\delta>0$ but it is in the infinitesimal vicinity of $X$ due to continuity.
	
	Now we can prove Proposition~\ref{discrete_spectral_measure}. Assume we have a point mass on $X\in\mathbb S_{\Herm(N)}$ and $X$ is neither one of $\pm I_N/\sqrt{N}$. Because the spectral measure is invariant, under a unitary conjugation $X$ can be transferred to another point mass by a unitary matrix $U\in\U(N)$, say $X'=UXU^\dagger$. Choosing the $U$ from above we see that $X'$ can be arbitrarily close to $X$ while $\delta\to0$, which means there are uncountably many $X'$ within a single orbit of $X$ on which the measure has to be non-vanishing. This is a contradiction to our assumption that there are only countably many point masses on $\mathbb S_{\Herm(N)}$.
	
	Yet, when $X=\pm I_N$, unitary conjugation does not change its value. Therefore the ensemble we described previously is the only possible ensemble induced by a Dirac point mass spectral measure.
\end{proof}

The question which remains is what the corresponding random matrix once we have chosen the spectral measure~\eqref{point.mass.spectral}. It is clear that the characteristic function is equal to
\begin{equation}\label{charac.Dirac.point}
\mathbb{E}[\exp(i \tr YS)]=\exp(-\gamma [p\nu_\alpha(\tr  S/\sqrt{N})+(1-p)\nu_{\alpha}(-\tr  S/\sqrt{N})]).
\end{equation}
However this equation already tells us what the random matrix is due to the uniqueness of the characteristic function

\begin{corollary}[Random Matrix of an Invariant Dirac Point Mass Spectral Measure]\label{cor:dirac.point.density}\

The invariant random matrix with the characteristic function~\eqref{charac.Dirac.point} is given by $Y=yI_N/\sqrt{N}$ with $y\in\R$ being a real random variable distributed by the L\'evy $\alpha$-stable distribution corresponding to the Fourier transform
\begin{equation}
\mathbb{E}[\exp(i yk)]=\exp(-\gamma [p\nu_\alpha(k)+(1-p)\nu_{\alpha}(-k)]).
\end{equation}
\end{corollary}

\begin{proof}
 The statement immediately follows by noting  $\exp(i \tr YS)=\exp[iy \tr S/\sqrt{N}]$. Then we only need to identify $k= \tr S/\sqrt{N}$.
\end{proof}

In conclusion, the Dirac spectral measure realises the simplest embedding of univariate probability theory in random matrix theory, namely a random variable times the identity matrix.

\subsection{Orbital Spectral Measure}\label{sec:orbit}

We have already seen that when we want to have a matrix $X_0\in\mathbb{S}_{\Herm(N)}$, which is not proportional to the identity matrix,  in the support of the spectral measure then the whole corresponding orbit $O_{X_0}$ under the conjugate action of $\U(N)$ has to be in its support. Thus it is rather natural to see orbits as the smallest supports of an invariant spectral measure and not Dirac point mass measures. Actually, the identity matrix and its negative from the previous subsection are the simplest kinds of orbits which only consist of a single element.

As the orbits are invariant under unitary conjugation, uniform measures on them are hence invariant, too. For now, let us construct the spectral measure containing a single orbit $\mathcal O_{X_0}$ with $X_0=\diag(x_1,\ldots,x_N)$. This measure $H_{X_0}$ is then explicitly defined by
\begin{equation}
	\int_{\mathbb S_{\Herm(N)}} \varphi(R)H_{X_0}(\dv R):=\la\varphi(UX_0U^\dagger)\ra,
\end{equation} 
for all bounded measurable functions $\varphi\in B(\mathbb S_{\Herm(N)})$. A comparison with~\eqref{5.1.1} with $R_m=\pm I_N$ shows that the Dirac point mass measure is indeed the simplest realisation of an orbital measure.

On the level of the eigenvalues any orbital measure is indeed a Dirac distribution that respects the permutation invariance, i.e., it is
\begin{equation}
	h_{\eig,X_0}(\dv r)=\frac{1}{N!}\sum_{\rho\in S_N}\prod_{j=1}^N\delta(r_j-x_{\rho(j)})\dv r_j.
\end{equation}
Employing~\eqref{3.1.10}, the characteristic function is explicitly
\begin{equation}
\mathbb E [\exp (i \tr  YU\diag(s)U^\dagger)]=\exp\left(-\gamma c_N\sum_{\rho\in S_N}\frac{\nu_\alpha(s^\top x_\rho)(s^\top x_\rho)^{N(N-1)/2}}{\Delta(s)\Delta(x_\rho)} \right),
\end{equation}
where $x_\rho=(x_{\rho(1)},\ldots,x_{\rho(N)})$.
Unfortunately, this function is too involved to easily evaluate its inverse spherical or its inverse Fourier transform to get an explicit representation of the corresponding joint probability density of the eigenvalues or diagonal entries, respectively.

This complication is also already reflected in the spectral measure $h_{\diag,X_0}(\dv t)$ for the diagonal entries. Exploiting the notation of Theorem~\ref{p3.1}, it is certainly
\begin{equation}
\tr  U{\rm diag}(s)U^\dagger X_0=s^\top\sum_{j=1}^Nx_j(|u_{j,1}|^2,\ldots,|u_{j,N}|^2)^\top 
\end{equation}
with a Haar distributed $U=\{u_{j,l}\}\in\U(N)$. Since we have
\begin{equation}
	\E [\exp(i\tr YS)]=\exp\left(-\gamma\int_{\|t\|\le 1}\nu_\alpha(s^\top t)h_{\diag,X_0}(\dv t)\right),
\end{equation}
it holds
\begin{equation}
	t=\sum_{j=1}^Nx_j(|u_{j,1}|^2,\ldots,|u_{j,N}|^2)^\top\qquad\Rightarrow\qquad \tr  U{\rm diag}(s)U^\dagger X_0=s^\top t.
\end{equation}
In the simplest case when the orbit only contains rank-one matrices, i.e., $X_0=\diag(\pm1,0,\ldots,0)$, we have $t=\pm(|u_{1,1}|^2,\ldots,|u_{1,N}|^2)^\top$ with a vector $(u_{1,1},\ldots,u_{1,N})$ uniformly distributed on the complex unit sphere $\{u\in\C^N:\|u\|=1\}=\mathbb S^{2N-1}$. Defining $t_l=\|u_{1,l}\|^2$,  in particular we take the polar decomposition for each vector component, we notice that the measure becomes the flat Lebesgue measure on the polytope $\mathcal{P}=\{t=(t_1,\ldots,t_N)^\top\in[0,1]^N| \sum_{l=1}^Nt_l=1\}$. Therefore, we can rewrite the characteristic function into the form
\begin{equation}\label{5.3.6}
	\E [\exp(i\tr YS)]=\exp\left(-\gamma\int_{\mathcal{P}}\nu_\alpha(s^\top t)\dv t\right).
\end{equation}

One can readily construct multiple-orbital spectral measure by a weighted sum of single-orbital measures. Given $m$ probability weights $p_j\in[0,1]$ with $\sum_{j=1}^m p_j=1$ which will be assigned to fixed diagonal matrices $X^{(1)},X^{(2)},\ldots, X^{(m)}$, respectively, we define the spectral measure via
\begin{equation}\label{orbital}
H_{\{X^{(1)},X^{(2)},\ldots, X^{(m)}\}}=\sum_{j=1}^m p_jH_{X^{(j)}}
\end{equation}
In other words, this spectral measure concentrates on separate orbits $\mathcal O_{X^{(j)}}$ with weights $p_j$, respectively, and each orbit in a uniform way. As a consequence, this gives rise to a sum of stable ensembles generated by those single-orbital measures $H_{X^{j}}$. The proof of the following proposition follows the same ideas as in~\cite[Ch. 1 Property 1.2.1]{ST94}.

\begin{proposition}\
	
	Let $H_{\{X^{(1)},X^{(2)},\ldots, X^{(m)}\}}$ and $H_{X^{(j)}}$ be defined as above. Let $Y_{\{X^{(1)},X^{(2)},\ldots, X^{(m)}\}}$ be drawn from the stable ensemble $S(\alpha,\gamma H_{\{X^{(1)},X^{(2)},\ldots, X^{(m)}\}},0)$ and $Y_{X^{(j)}}$ be independently drawn from the stable ensembles $S(\alpha,\gamma H_{X^{(j)}},0)$. Then for $\alpha\ne 1$,
	\begin{equation}\label{multip_orbit}
	Y_{\{X^{(1)},X^{(2)},\ldots, X^{(m)}\}}=\sum_{j=1}^m p_j^{1/\alpha}Y_{X^{(j)}}.
	\end{equation}
	For $\alpha=1$, it is instead
	\begin{equation}\label{multip_orbit1}
		Y_{\{X^{(1)},X^{(2)},\ldots, X^{(m)}\}}=\sum_{j=1}^m p_j\left(Y_{X^{(j)}}-\frac{2\gamma\tr X^{(j)}}{\pi N}\log (p_j) I_N\right).
	\end{equation}
\end{proposition}

\begin{proof}
	For $\alpha\ne 1$, we start with the spectral measure~\eqref{orbital} and compute
	\begin{equation}\label{5.3.12.b}
	\begin{split}
		\int_{\mathbb S_{\Herm(N)}}\sum_{j=1}^m p_jH_{X^{(j)}}(\dv R)\nu_\alpha(\tr SR)=\sum_{j=1}^m\int_{\mathbb S_{\Herm(N)}} H_{X^{(j)}}(\dv R)\nu_\alpha( p_j^{1/\alpha}\tr SR).
	\end{split}
	\end{equation}
	This equality is a consequence of the scaling property $k\nu_\alpha(x)=\nu_\alpha(k^{1/\alpha}x)$, for any $k>0$. The right hand side is indeed the characteristic function of $Y_{\{X^{(1)},X^{(2)},\ldots, X^{(m)}\}}$.
	
	For $\alpha=1$, the scaling property instead becomes
	\begin{equation}
		k\nu_1(x)=\nu_1(kx)-\frac{2i}{\pi}x k\log k,\ {\rm for\ any}\ k>0.
	\end{equation}
	Therefore,
	\begin{equation}\label{5.3.13}
	\begin{split}
		&\int_{\mathbb S_{\Herm(N)}}p_jH_{X^{(j)}}(\dv R)\nu_\alpha(\tr SR)\\
		=&\int_{\mathbb S_{\Herm(N)}} H_{X^{(j)}}(\dv R)\nu_\alpha(p_j\tr SR)-\frac{2i}{\pi} p_j\log p_j \tr \left(S\int_{\mathbb S_{\Herm(N)}} RH_{X^{(j)}}(\dv R)\right).
	\end{split}
	\end{equation}
	For the second integral on the right hand, we have
	\begin{equation}
		\int_{\mathbb S_{\Herm(N)}} RH_{X^{(j)}}(\dv R)=\la UX^{(j)}U^\dagger\ra=\frac{\tr X^{(j)}}{N}I_N.
	\end{equation}
	When putting everything together into a calculation similar to~\eqref{5.3.12.b} we find the statement for $\alpha=1$.
\end{proof}

	The number $m$ of matrices $X^{(j)}$ or equivalently of their random matrix counterparts $Y_{X^{(j)}}$  can be essentially taken to infinity. The convergence holds for both sums in~\eqref{multip_orbit} and~\eqref{multip_orbit1}. This can be seen via the spectral measure as it approximates arbitrary stable distributions, compare with~\cite{BNR93}. This can be exploited in simulations where a random matrix $Y_{\{X^{(1)},X^{(2)},\ldots, X^{(m)}\}}$ may be used to approximate any invariant stable ensemble.

\section{The Generalised Central Limit Theorem}\label{s4}

\subsection{Domain of Attractions for $\alpha\in(0,2)$}\label{sec:stable}

The stable law (also known as generalised central limit theorem) is an analogue of the classical central limit theorem, in the case when the random variable does not have a finite variance. It is important in developing a stable law, to specify its \textit{domain of attraction}. In~\cite{GKC}, the domain of attraction is determined for each stable distribution so that whenever a random variable falls into such domains, its scaled (and possibly shifted) average converges to the corresponding stable distribution.  The domain of attraction is formally defined by statement (1) in Theorem~\ref{multi_stable_law}.

\begin{theorem}[Multivariate Central Limit Theorem; see ~\cite{Rvaceva,Shimura}]\label{multi_stable_law}\
	
	Let $S(\alpha,h,0)$ be an $\R^{d}$ stable distribution with a normalised spectral measure $h$ and index $0<\alpha<2$. Let $F(\dv x)$ denote the probability measure of $S(\alpha,h,0)$. The following statements are equivalent:
	\begin{enumerate}
		\item[(1)] an $\R^d$ random vector $x$ belongs to the domain of attraction of $S(\alpha, h,0)$, that is, there are sequences of positive real numbers $\{B_m\}_{m\in\mathbb{N}}\subset\R_+$ and $d$-dimensional vectors $\{A_m\}_{m\in\mathbb{N}}\subset\R^d$ for a sequence of random vectors $\{x_m\}_{\m\in\mathbb{N}}\in\R^d$ of independent copies of $x$ such that it holds
		\begin{equation}\label{4.1.1}
			\lim_{m\to\infty}\left(\frac{x_1+\ldots+x_m}{B_m}-A_m\right)= y
		\end{equation}
		in distribution with $y$ a random vector sampled from $S(\alpha, h,0)$;
		\item[(2)] for any $k>0$ and $\Sigma$ being a $h$-continuous Borel  measurable set on the unit sphere $\mathbb S^{d-1}\subset\R^d$  it holds
		\begin{equation}\label{4.1.2}
			\lim_{R\to\infty}\frac{\mathbb P(||x||>kR,\,x/||x||\in \Sigma)}{\mathbb P(||x||>R)}=  k^{-\alpha}\mathbb \int_{s\in \Sigma} h(\dv s);
		\end{equation}
		\item[(3)] for any $k>0$ and $\varphi$ being a continuous  function on the unit sphere $\mathbb S^{d-1}\subset\R^d$  it holds
		\begin{equation}\label{4.1.3}
			\lim_{R\to\infty}\frac{\int_{\|x\|>kR}\varphi(x/\|x\|)F(\dv x)}{\mathbb P(||x||>R)}=  k^{-\alpha}\mathbb \int_{s\in \mathbb S^{d-1}} \varphi(s)h(\dv s).
		\end{equation}
	\end{enumerate}
\end{theorem}

Let us emphasise for point (3) that continuity implies boundedness for the function $\varphi$ due to compactness of the sphere.

\begin{proof}
	The equivalence of statements (1) and (2) is given by~\cite{Rvaceva}.

	The equivalence of statements (2) and (3) becomes straightforward with the help of Portemanteau's theorem (see, e.g., equivalence between (i) and (vi) in~\cite[Theorem 13.16]{Klenke}). Indeed by Portemanteau's theorem, statement (2) is saying that for all $k>0$, the measure $\mu^{(k)}_R$ converges weakly to $k^{-\alpha}h$ when $R$ goes to $\infty$ where the measures $\mu^{(k)}_R$ is defined for any Borel set $\Sigma\subset \mathbb S^{d-1}$ as
	\begin{equation}
		\mu_R^{(k)}(A)=\frac{\mathbb P(||x||>kR,\,x/||x||\in A)}{\mathbb P(||x||>R)},
	\end{equation}
	which can be restated exactly as statement (3), by using integration over continuous functions $\varphi$ (the boundedness of $\varphi$ is guarantedd by the compactness of its domain $\mathbb S^{d-1}$).
\end{proof}

Let us underline that also the domains of attraction of the stable ensembles $S(\alpha,\gamma h,y_0)$ , especially with a fixed scaling $\gamma>0$ and a fixed shift $y_0\in\R^d$ are covered by Theorem~\ref{multi_stable_law}. These two additional parameters are absorbed in the sequences $\{B_m\}_{m\in\mathbb{N}}\in\R_+$ and $\{A_m\}_{m\in\mathbb{N}}\in\R^d$, respectively. The challenge is to derive these sequences from a given probability measure $F(\dv x)$. For instance, it holds $\mathbb{E}[x]=A_mB_m/m$ when $\alpha\in(1,2)$. Once $A_m$ is computed in this way one can get $B_m$. However, for a stability exponent $\alpha\leq1$ the first moment does not exist so that one needs to find other means. 

Theorem~\ref{DoA} is essentially a rephrasing of the theorem above into a form for invariant random matrix ensembles on $\Herm(N)$, in the case $\alpha<2$. The additional content lies in the implication of the domains of attraction for the eigenvalues and the diagonal entries of the random matrix. The scaling parameter $\gamma>0$ and the shift $y_0\in\mathbb{R}$ originate in a non-optimal choice of $B_m$ and $A_m$, respectively. Then, the limiting random matrix $Y$ in~\eqref{def.CLT} is drawn from $S(\alpha, \gamma H,y_0I_N)$, or equivalently the random vector is $\gamma^{1/\alpha}Y+y_0I_N$ with $Y$ drawn from $S(\alpha, H,0)$ when $\alpha\neq1$. The latter shows that the domain of attraction is for $S(\alpha, H,0)$ and $S(\alpha, \gamma H,y_0I_N)$ the same as they only differ by a scaling and a shift. The case $\alpha=1$ is more subtle and will be discussed at the end of this section.

One should note that statement (3) of Theorem~\ref{DoA} means the sum the random vector of eigenvalues turns into the multidimensional stable distribution $S(\alpha,h_\eig,0)$. This is completely different from statement (1) where we add up matrices instead. Hence, $y$ is not the collection of all eigenvalues of the random matrix $Y$ in statement (1). Also the limiting distributions in (1) and (3) are different (c.f.~\eqref{3.1.10} where in the former the integral with measure $h_\eig$ contains terms other than $\nu_\alpha(s^\top r)$). Therefore, it comes more like a surprise that understanding the eigenvalues as an ordinary random vector in $\R^N$ which we we add up in the standard way like a vector yields the very same spectral measure $h_\eig$ which we also encounter when understanding them as the eigenvalues of Hermitian random matrices.

We believe that the additional statement about $C_m=1$ for all $m\in\mathbb{N}$ can be extended also to the case when the traces vanish. The reason is that the difference between $X$ and $x_\eig$ lies in the Haar distributed eigenvectors which are bounded random variables. We do not expect that they change anything in these constants. However, we were unable to prove this in the general setting. As the proof will show, we need a non-vanishing trace to fix essentially the scaling $C_m$ since, when taking the trace of~\eqref{conv_in_distr} and~\eqref{conv_in_distr.c}, we get expressions that are equal on the left hand side of the two equations apart from the $C_m$. When the trace vanishes we essentially multiply $C_m$  with a term that vanishes faster when $m\to\infty$ than $C_m$ can grow. This implies that the scale, on which the distribution of the trace has its support, is much smaller than the one of the traceless part of the random matrix $X$.

Before we prove Theorem~\ref{DoA} we need the following Lemma which is essentially a rewriting of~\eqref{4.1.2}. It will serve as one starting point  in the proof of some equivalences claimed in Theorem~\ref{DoA}.

\begin{lemma}[]\label{lem:invariance}\
	
	Considering the setting of Theorem~\ref{multi_stable_law} and let $\mathcal{G}$ be a compact matrix Lie group or a finite group, acting on $\R^d$ as isometries, i.e., $||gx||=1$ for all $x\in\R^d$ and $g\in\mathcal{G}$, and assume that the probability measure $F(\dv x)$ is invariant under $\mathcal{G}$, meaning $F(\dv x)=F(\dv (gx))$  for all $g\in\mathcal{G}$. Then, the statements of Theorem~\ref{multi_stable_law} implies $h$ being also invariant. 
	
	Also statement (3) of Theorem~\ref{multi_stable_law} is equivalent to the following statement:
	\begin{enumerate}
		\item[(3*)] for any $k>0$ and $\varphi$ being a continuous function on the unit sphere $\mathbb S^{d-1}\subset\R^d$, which is also invariant under $\mathcal G$, equation~\eqref{4.1.3} holds.
	\end{enumerate}
\end{lemma}

\begin{proof}
	To see the invariance of $h$, we take an arbitrary continuous function $\psi$ on $\mathbb S^{d-1}$ and by the invariance of $F$ we have
	\begin{equation}
		\int_{\|x\|>kR}\psi(x/\|x\|)F(\dv x)=\int_{\|gx\|>kR}\psi(gx/\|x\|)F(\dv gx)=\int_{\|x\|>kR}\psi(gx/\|x\|) F(\dv x)
	\end{equation}
	for any $g\in\mathcal G$ because of the group invariance of the measure $F$ and that the group elements are isometries. Hence, we have by statement (3) of Theorem~\ref{multi_stable_law} for any continuous function $\psi$ on $\mathbb S^{d-1}$,
	\begin{equation}
		\begin{split}
			k^{-\alpha}\mathbb \int_{s\in \mathbb S^{d-1}} \psi(s) h(\dv s)&=\lim_{R\to\infty}\frac{\int_{\|x\|>kR}\psi(x/\|x\|)F(\dv x)}{\mathbb P(||x||>R)}=\lim_{R\to\infty}\frac{\int_{\|x\|>kR}\psi(gx/\|x\|) F(\dv x)}{\mathbb P(||x||>R)}\\
			&=k^{-\alpha}\mathbb \int_{s\in \mathbb S^{d-1}} \psi(gs) h(\dv s)=k^{-\alpha}\mathbb \int_{s\in \mathbb S^{d-1}} \psi(s) h(g^{-1}\dv s)
		\end{split}
	\end{equation}
	which shows the weak identity $h(\dv s)=h(g^{-1}\dv s)$ for all $g\in\mathcal G$.
	
	Now we prove the equivalence of statement (3) of Theorem~\ref{multi_stable_law} and statement (3*) of Lemma~\ref{lem:invariance}. Certainly statement (3) of Theorem~\ref{multi_stable_law} implies that also group invariant continuous functions satisfy~\eqref{4.1.3}. Thus, we only need to prove the converse.
	
	We assume that Eq.~\eqref{4.1.3} is up to now only satisfied for continuous functions invariant under the group action of $\mathcal G$. Let $\dv g$ be the Haar measure for the Lie group or the uniform Dirac measure for the finite group, respectively. When choosing an arbitrary continuous function $\varphi$ on $\mathbb S^{d-1}$. Then,
	\begin{equation}
		\la \varphi(gx/\|x\|)\ra:=\int_{\mathcal G}\dv g\varphi(gx/\|x\|)
	\end{equation}
	is a continuous function, too, since the measure $\dv g$ is finite. Therefore, we have again
	\begin{equation}
		\int_{\|x\|>kR}\varphi(x/\|x\|)F(\dv x)=\int_{\|x\|>kR}\varphi(gx/\|x\|)F(\dv x)=\int_{\|x\|>kR}\la\varphi(gx/\|x\|)\ra F(\dv x).
	\end{equation}
	This implies for each continuous function $\varphi$ that we can integrate it over the group G yielding $\la\varphi(gx)\ra$; this is again a continuous function on $\mathbb S^{d-1}$. Then
	\begin{equation}
		\begin{split}
		\lim_{R\to\infty}\frac{\int_{\|x\|>kR}\varphi(x/\|x\|)F(\dv x)}{\mathbb P(||x||>R)}&=\lim_{R\to\infty}\frac{\int_{\|x\|>kR}\la\varphi(gx/\|x\|)\ra F(\dv x)}{\mathbb P(||x||>R)}\\
		&=  k^{-\alpha}\mathbb \int_{s\in \mathbb S^{d-1}} \la\varphi(gs)\ra h(\dv s)=k^{-\alpha}\mathbb \int_{s\in \mathbb S^{d-1}} \varphi(s) h(\dv s).
		\end{split}
	\end{equation}
	for an arbitrary continuous function $\varphi$ which has been the goal to show.
\end{proof}

\begin{proof}[Proof of Theorem~\ref{DoA}]
	
		First, we prove (1)$\Rightarrow$(2) which is the simplest implication. We denote the  corresponding probability measures on both sides of~\eqref{conv_in_distr} by $F^{(m)}$ and $F^{(\infty)}$, respectively, and the induced marginal measures for the diagonal entries $x_{\diag}$ of the random matrix $X$ by $f_{\diag}^{(m)}$ and $f^{(\infty)}_{\diag}$. Then, this equation can be reinterpreted in the distributional sense. This means it holds
		\begin{equation}\label{4.2.2}
			\lim_{m\rightarrow \infty}\int_{\Herm(N)} \Phi(X) F^{(m)}(\dv X)=\int_{\Herm(N)} \Phi(X) F^{(\infty)}(\dv X)
		\end{equation}
		for any bounded measurable function $\Phi$ on $\Herm(N)$. When we take a bounded measurable function $\psi$ on $\mathbb{R}^N$ and choose $\varphi(X)=i_2\psi(X)$, see~\eqref{def.i2}, statement (2) follows in the form
		\begin{equation}\label{4.2.3}
			\lim_{m\rightarrow \infty}\int_{\R^N} \psi(x_{\diag}) f_{\diag}^{(m)}(\dv x_{\diag})=\int_{\R^N} \Phi(x_{\diag}) f^{(\infty)}_{\diag}(\dv x_{\diag})
		\end{equation}
		as we only need to integrate over all off-diagonal entries. The two subsequences $\{B_m\}_{m\in\mathbb{N}}\subset\R_+$ and $\{A_m\}_{m\in\mathbb{N}}\subset\R$ are the same as we note that the diagonal entries of the left hand side of~\eqref{conv_in_distr} are given by the left hand side of~\eqref{conv_in_distr.b}. Combining this inside with Theorem~\ref{p3.1a}, in particular Eq.~\eqref{g_alpha}, for the relation of the spectral measure of the random matrix and the one for the diagonal entries we find that the measure $f^{(\infty)}_{\diag}$ over $\R^N$ generates the ensemble $S(\alpha,g_\alpha,y_1)$.
		
		To prove the converse (2)$\Rightarrow$(1), we note that the two sequences $\{B_m\}_{m\in\mathbb{N}}\subset\R_+$ and $\{A_m\}_{m\in\mathbb{N}}\subset\R$ are the same because of the very same reason as above. The goal is to reconstruct the random matrices $X_j$ and $Y$ out of the information about the diagonal entries. For this purpose we need to overcome the problem that we have no explicit inverse of the map $i_2$. The derivative principle~\eqref{p2.2} and the uniqueness of the Haar measure gives us a way to do this reconstruction, nonetheless. We first define the joint probability measure of the vectors $x_{\eig,j}$ and $y_{\eig}$ which we can interpret as the eigenvalues of the random matrices $X_j=U_j\diag(x_{\eig,j})U_j^\dagger$ and $Y=U\diag(y_{\eig})U^\dagger$ with independent Haar distributed matrices $U,U_1,U_2,\ldots\in\U(N)$. The limit is then given because the characteristic functions restricted to diagonal matrices $S={\rm diag}(s)$ agree with each other, and the limit of the former exists. The spectral measure $H$ can be extracted with the help of Theorems~\ref{p3.1} and~\ref{p3.1a}. 
		
		Next, we turn to the equivalence of (3) to the other two statements for which we need  the invariance of the random matrix $X$ under the conjugate action of $\U(N)$. This invariance implies a permutation invariance of the diagonal entries $x_{\diag}$ and eigenvalues $x_{\diag}$. Moreover, the conjugate action of $\U(N)$ is an isometry on $\Herm(N)$ and the symmetric group $S_N$ acting as permutations of the vector entries is an isometric action on $\R^N$. Thus, we can apply Lemma~\ref{lem:invariance}.
		
		Let us first show the implication of the spectral measure $H$ in (1) to the spectral measure in (3) which also proves the existence of the limit. By Theorem~\ref{multi_stable_law} part (3) and Lemma~\ref{lem:invariance}, Eq.~\eqref{conv_in_distr} is equivalent to
		\begin{equation}\label{4.2.6}
			\lim_{R\to\infty}\frac{\int_{\sqrt{\tr X^2}>kR}\Phi\left({X}/{\sqrt{\tr X^2}}\right)F(\dv X)}{\int_{\sqrt{\tr X^2}>R}F(\dv X)}= k^{-\alpha}\int_{\mathbb S_{\Herm(N)}}\Phi(S)H(\dv S)
		\end{equation}
		for any continuous  function $\Phi$ on $\mathbb S_{\Herm(N)}$ on the sphere $\mathbb S_{\Herm(N)}$  that is invariant under the conjugate action of $\U(N)$. When choosing an arbitrary continuous  function $\phi$ on $\mathbb S_{\Herm(N)}$ on the sphere $\mathbb S^{N-1}\subset\R^N$, we can choose $\Phi=i_1\phi$ which is defined by Eq.~\eqref{i_1} with Haar distributed eigenvectors of the random matrix $X$. Then, certainly $\Phi$ satisfies the condition of being bounded, measurable and invariant. We plug this particular $\Phi$ into~\eqref{4.2.6} and diagonalise the random matrix $X=U\diag(x_{\rm eig})U^\dagger$ as well as $S=V\diag(s_{\rm eig})V^\dagger$ and exploit that $U,V\in\U(N)/\U^N(1)$ are Haar distributed. This yields
		\begin{equation}\label{4.1.13}
			\lim_{R\to\infty}\frac{\int_{||x_{\rm eig}||>kR}\phi\left({x_{\rm eig}}/{||x_{\rm eig}||}\right)f_\eig(\dv x_{\rm eig})}{\int_{||x_{\rm eig}||>R}f_\eig(\dv x_{\rm eig})}= k^{-\alpha}\int\phi(s_{\rm eig})h_\eig(\dv s_{\rm eig}),
		\end{equation}
		where $f_\eig$ and $h_\eig$ are the induced measures of $F$ and $H$ for the eigenvalues of $X$ and $S$.  By Theorem~\ref{multi_stable_law} and Lemma~\ref{lem:invariance} this is equivalent with $x_{\eig}$ being a random vector in the domain of attraction of the ensemble $S(\alpha,h_\eig,0)$. 
		
		The statement about the sequences $\{A_m\}_{m\in\mathbb{N}}$ and  $\{B_m\}_{m\in\mathbb{N}}$ follows from the trace which is for all three objects $X$, $x_\diag$ and $x$ the same,
		\begin{equation}
			\tr X=\tr \diag(x_{\diag})=\tr \diag(x_{\eig}),
		\end{equation}
		as we can take the trace of any of the three equations~\eqref{conv_in_distr},~\eqref{conv_in_distr.b}, and~\eqref{conv_in_distr.c}. The trace is itself a univariate real random variable and each of the three equations say how to shift and rescale the sum of these random variables to find the univariate generalised central limit theorem. The shift $A_m$ and the rescaling $B_m$ are unique in the leading $m$ asymptotics when the scaling is $\gamma=1$ and the shift is $y_0=0$ of the corresponding stable distribution. This immediately yields the additional statement that we can choose $C_m=1$  for all $m\in\mathbb{N}$ when the traces $\tr  Y=\tr \diag(y_{\diag})=\tr \diag(y)\neq0$ do not vanish. 
		
		When the traces vanish it means that the distribution of the trace is given by  a Dirac delta distribution at the origin because the shift in the ensemble is assumed to be vanishing, i.e., $y_0=0$ for $S(\alpha,H,y_0I_N)$ and $S(\alpha,h_{\eig},y_01_N)$. Thence, the rescaling with respect to $C_m$, respectively, cannot be fixed with the help of the trace as it is multiplied to a vanishing random variable (which  is the trace). This means we cannot draw any conclusions as the scaling $C_m$ is dominated by the traceless parts of $\sum_{j=1}^mX_j$ and $\sum_{j=1}^m x_{\eig,j}$ which are irreducible representations of the conjugate action of the unitary group and of the permutation group, respectively, especially they do not mix with the trivial representations of these two groups which is given by the trace. This also implies that the shift by $A_m$ is solely due to the traces of $\sum_{j=1}^mX_j$ and $\sum_{j=1}^m x_{\eig,j}$, respectively.
		
		To prove the converse implication (3)$\Rightarrow$(1), we note that the map $i_1$ is a bijection between the set of continuous  functions on $\mathbb{S}_{\Herm(N)}$ invariant under the conjugate action of $\U(N)$ and the continuous  functions on $\mathbb{S}^{N-1}$ invariant under the symmetric group $S_N$. To see this we underline that any invariant continuous  function $\Phi$ on $\mathbb S_{\Herm(N)}$ only depends on the eigenvalues $x_{\eig}\in\mathbb{S}^{N-1}$ of a matrix $X\in\mathbb S_{\Herm(N)}$, i.e., for $X=U\diag(x_{\rm eig})U^\dagger$ with $U\in\U(N)$ it holds $\Phi(X)=\Phi(\diag(x_{\rm eig}))$. Defining $\phi(x_\eig)=\Phi(\diag(x_{\rm eig})$ yields an invariant continuous  function on $\mathbb{S}^{N-1}$ and it holds $i_1\phi=\Phi$  giving the surjectivity of $i_1$. the injectivity follows because $i_1$ is linear and applied on the zero function it yields again a zero function, anew because of~\eqref{i_1}.
		
		Thus, once we  exploit Theorem~\ref{multi_stable_law} and Lemma~\ref{lem:invariance}  to~\eqref{conv_in_distr.c} we can write for an arbitrary continuous function $\phi$ on $\mathbb{S}^{N-1}$ invariant under the symmetric group $S_N$ the integral~\eqref{4.1.13}. Moreover, we can choose $\phi=i_1^{-1}\Phi$ where $\Phi$ on $\mathbb S_{\Herm(N)}$ is an arbitrary measurable function that is bounded and invariant under the conjugate action of $\U(N)$. Explicitly, this means
		\begin{equation}\label{4.1.13.c}
			\lim_{R\to\infty}\frac{\int_{||x_{\rm eig}||>kR}\Phi\left(\diag(x)/{||x_{\rm eig}||}\right)f_\eig(\dv x)}{\int_{||x_{\rm eig}||>R}f_\eig(\dv x)}= k^{-\alpha}\int\Phi(\diag(s_{\rm eig}))h_\eig(\dv s_{\rm eig}),
		\end{equation}
		Because of the invariance we can introduce a Haar distributed unitary matrix as follows
		\begin{equation}
			\Phi\left(\diag(x)/{||x_{\rm eig}||}\right)=\int_{\U(N)/\U^N(1)}\dv U \Phi\left(U\diag(x)U^\dagger/{||x_{\rm eig}||}\right)
		\end{equation}
		and similarly for the right hand side of~\eqref{4.1.13.c}. For the measures we have the relations $F(\dv X)=f_{\rm eig}(\dv x_\eig)\dv U$ and $H(\dv S)=h_{\rm eig}(\dv s_\eig)\dv V$ when considering the decompositions  $X=U\diag(x_{\rm eig})U^\dagger$ as well as $S=V\diag(s_{\rm eig})V^\dagger$ and exploit that $U,V\in\U(N)/\U^N(1)$. In this way we end up with~\eqref{4.2.6} which is equivalent with statement (1) (with possibly different sequences $\{A_m\}_{m\in\mathbb{N}}$ and  $\{B_m\}_{m\in\mathbb{N}}$) because of Theorem~\ref{multi_stable_law} and Lemma~\ref{lem:invariance}.
		
		The sequences $\{A_m\}_{m\in\mathbb{N}}$ and  $\{B_m\}_{m\in\mathbb{N}}$ are the same as those in statement (3) because of the same reason as before, namely that the random variable of the trace is in all three statements the same.
		This concludes the proof.
\end{proof}
	
\subsection{Strict Domain of Attraction}\label{sec:strict}

We want to  briefly discuss the strict domain of attraction of stable invariant distributions. The definition is as follows.

\begin{definition}[Strict Domain of Attraction~\cite{Shimura}]\label{def:strict.attract}\

The random vector $x\in\R^d$ in Theorem~\ref{multi_stable_law}(1) is in the strict domain of attraction of the ensemble $S(\alpha,h,0)$ iff $A_m=0$.
\end{definition}

In~\cite{Shimura} such strict domains of attraction have been studied. For $\alpha>1$, the probability measure has only to satisfy $\mathbb EX=0$ which is rather clear as the first moment still exists, see~\cite[Proposition 4.2]{Shimura}.

For $0<\alpha<1$, each random vector which is in the domain of attraction of a stable distribution, especially satisfying~\eqref{4.1.2}, is also in the strict domain of attraction~\cite[Theorem 3.3]{Shimura}. This can be readily seen when choosing an arbitrary stable distribution with stability exponent $\alpha\in(0,1)$ and then take the sum of $m$ copies of the corresponding random vector. The scaling parameter is $B_m=m^{1/\alpha}$. Exponentiating its characteristic function~\eqref{3.1.1} by $m$ due to the $m$-fold convolution of the distribution and then rescaling $s\to s/m^{1/\alpha}$ we notice that the linear term $iv_0^Ts$ in~\eqref{3.1.1} gets a prefactor $m^{1-1/\alpha}$ which vanishes when taking the limit $m\to\infty$. Thus, even the stable distributions that are not strictly stable become strictly stable when we do not shift but only rescale the sum. Hence, it is no surprise that this also holds for the other random vectors that are in the domain of attraction for stable distributions with $\alpha\in(0,1)$.

 For $\alpha=1$, the sufficient and necessary condition (additional to~\eqref{4.1.2} for being in the strict domain of attraction is
\begin{equation}\label{strict}
	\lim_{R\to\infty}\frac{\int_{\|x\|<R}xF(\dv x)}{R\int_{\|x\|>R}F(\dv x)}=0.
\end{equation}
We underline that we set here $y_0=0$ and $\gamma=1$ in contrast to~\cite{Shimura} which can be always arranged after an appropriate scaling with $B_m$ and shift with $A_m$.

In the setting of $X\in\Herm(N)$ being an invariant random matrix the condition~\eqref{strict} for $\alpha=1$ reads
\begin{equation}\label{strict.matrix}
	\lim_{R\to\infty}\frac{\int_{\sqrt{\tr X^2}<R}XF(\dv X)}{R\int_{\sqrt{\tr X^2}>R}F(\dv X)}=0.
\end{equation}
Considering only the diagonal entries it is immediate what the condition for the vector $x_\diag$ looks like, namely
\begin{equation}\label{strict.matrix.diag}
	\lim_{R\to\infty}\frac{\int_{\|x_\diag\|<R}x_\diag f_\diag(\dv x_\diag)}{R\int_{\|x_\diag\|>R}f_\diag(\dv x_\diag)}=0.
\end{equation}
For the eigenvalues we decompose the random matrix $X=U\diag(x_\eig)U^\dagger$ into its Haar distributed eigenvectors $U\in\U(N)/\U^N(1)$ and its eigenvalues $x_\eig\in\R^N$. Then we can integrate over $U$ and find
\begin{equation}\label{strict.matrix.eig.c}
	\lim_{R\to\infty}\frac{\int_{\|x_\eig\|<R}\tr \diag(x_\eig) f_\eig(\dv x_\eig)}{R\int_{\|x_\eig\|>R}f_\eig(\dv x_\eig)}=0.
\end{equation}
It is the trace due to the average $\langle U\diag(x_\eig)U^\dagger\rangle=[\tr \diag(x_\eig)]I_N/N$. However, the permutation invariance of the marginal measure $f_\eig(\dv x_\eig)$ tells us that it is equivalent with
\begin{equation}\label{strict.matrix.eig.b}
	\lim_{R\to\infty}\frac{\int_{\|x_\eig\|<R} x_\eig f_\eig(\dv x_\eig)}{R\int_{\|x_\eig\|>R}f_\eig(\dv x_\eig)}=\frac{I_N}{N}\lim_{R\to\infty}\frac{\int_{\|x_\eig\|<R} \tr \diag(x_\eig) f_\eig(\dv x_\eig)}{R\int_{\|x_\eig\|>R}f_\eig(\dv x_\eig)}=0.
\end{equation}
Thus, it again agrees with the condition derived in~\cite{Shimura} as it should due to Theorem~\ref{DoA}.

\subsection{The Gaussian Case ($\alpha=2$)}\label{sec:Gauss}

The case $\alpha=2$ is as particular as the case $\alpha=1$. It is the case of Gaussian probability distributions. We will anew prove the equivalence of domains of attraction for the random matrix, its diagonal entries and its eigenvalues.

Let us briefly recall what situation in the multivariate case a random vectors $x$ distributed along the Borel probability measure $F$ belongs to the domain of attraction of a normal distribution with characteristic function $\exp(-\gamma s^\top\Sigma s/2)$ with $\gamma$ a scaling parameter and $\Sigma$ the covariance matrix, if and only if~\cite{Rvaceva,Shimura}
\begin{eqnarray}\label{4.2.1}
	\lim_{R\to0}\frac{R^2\int_{\|x\|>R}F(\dv x)}{\int_{\|x\|<R}\|x\|^2F(\dv x)}&=& 0,\\
\label{4.2.2a}
	\lim_{R\to0}\frac{\int_{\|x\|<R}(t^\top x)^2F(\dv x)}{\int_{\|x\|<R}(s^\top x)^2F(\dv x)}&=&\frac{t^\top\Sigma t}{s^\top\Sigma s}\qquad {\rm for\ all}\ s,t\in\mathbb{R}^d.
\end{eqnarray}
We would like to point out that the overall scaling $\gamma>0$ is not fixed by this classification. The reason why this is done in this way is that it can be always absorbed in the rescaling sequence $\{B_m\}_{m\in\mathbb{N}}$. Thus we can set it equal to $\gamma=1$. 
For most probability distributions one would compute the second moment to obtain the covariance matrix $\Sigma$. However, there are distributions, like the one with the density proportional to $1/(1+\|x\|^2)^{(d+2)/2}$ with $x\in\R^d$, which have no second moment but still converge to a Gaussian distribution. This is the reason why one needs to express these non-trivial conditions in terms of limits like in the heavy-tailed case with $\alpha<2$.

In the case of invariant Hermitian random matrices, we prove the following classification of the domain of attraction to a centred Gaussian which replaces Theorem~\ref{DoA} for $\alpha<2$. Note that the centred Gaussians are the strictly stable distributions for $\alpha=2$. For invariant Hermitian random matrices, the conditions~\eqref{4.2.1} and~\eqref{4.2.2a} take the form
\begin{equation}\label{4.2.1.b}
	\lim_{R\to0}\frac{R^2\int_{\sqrt{\tr X^2}>R}F(\dv X)}{\int_{\sqrt{\tr X^2}<R}\tr X^2F(\dv X)}= 0,\quad
	\lim_{R\to0}\frac{\int_{\sqrt{\tr X^2}<R}(\tr  T X)^2F(\dv X)}{\int_{\sqrt{\tr X^2}<R}(\tr  S X)^2F(\dv X)}=\frac{\sigma^2\tr  T^2+\kappa(\tr T)^2}{\sigma^2\tr  S^2+\kappa(\tr S)^2}
\end{equation}
for all $S,T\in\Herm(N)$. Representation theory dictates that there are essentially only two parameters that can create the covariance matrix, the trace and the traceless part of the Hermitian matrices. This will be at the heart of the proof for one part of the ensuing theorem.

\begin{theorem}[Gaussian Domain of Attraction]\label{DoA2}
	Let $\alpha=2$, and the characteristic function of the stable invariant ensemble $Y$ and, equivalently, of its diagonal entries has the characteristic function $\exp[-\sigma^2\|s\|^2/2-\kappa(s^\top1_N)^2/2]$ with $\sigma>0$ and $\kappa>-\sigma^2/N$. With the notations used in  Theorem~\ref{DoA}, the following statements are equivalent:
	\begin{enumerate}
		\item the invariant random matrix $X\in\Herm(N)$ belongs to the domain of attraction of the ensemble $Y$;
		\item the diagonal entries $x_\diag$ of $X$ belong to the domain of attraction of $y_\diag$;
		\item the eigenvalues $x_\eig$ of $X$ belong to the domain of attraction of the Gaussian with the characteristic function $\exp[-(N+1)\sigma^2\|s\|^2/2-(\kappa-\sigma^2)(s^\top1_N)^2/2]$, where we take the sum of $m\to\infty$ independent copies of $x_{\eig}$ like in Theorem~\ref{DoA}(3).
	\end{enumerate}
\end{theorem}

Again let us emphasise that the sum of the eigenvalues of the $m$ independent copies of Theorem~\ref{DoA2}(3) does not yield the eigenvalues of the sum of the $m$ independent copies of the random matrix $X$. It is, thence, not surprising that the limit does not follow the joint probability density~\eqref{alpha=2a} but  the Gaussian distribution of the diagonal entries.

\begin{proof}
	Again, (1)$\Rightarrow$(2) follows from the same reasoning as in the proof of Theorem~\ref{DoA}, namely we only need to take the diagonal entries of~\eqref{conv_in_distr}. The covariance follows from choosing diagonal matrices $T$ and $S$ in the second condition in~\eqref{4.2.1.b} and integrating over all non-diagonal entries yielding the marginal measure $f_\diag$ for the diagonal entries.  The opposite direction (2)$\Rightarrow$(1) is again a direct consequence of the derivative principle~\eqref{2.3.5} and its uniqueness~\eqref{2.3.6}.

	 To show the equivalence of (1) and (3), we split the three matrices in~\eqref{4.2.1.b} $T=T_{\rm l}+T_{\rm t}I_N/N$,  $S=S_{\rm l}+S_{\rm t}I_N/N$, and  $X=X_{\rm l}+X_{\rm t}I_N/N$ in their traces $T_{\rm t}=\tr  T$, $S_{\rm t}=\tr  S$, and $X_{\rm t}=\tr  X$ and their traceless parts $T_{\rm l},S_{\rm l},X_{\rm l}$, respectively. The same we can do for the corresponding eigenvalues $t_\eig=t_{\eig,\rm l}+t_{\eig,\rm t} 1_N/N$, $s_\eig=s_{\eig,\rm l}+s_{\eig,\rm t} 1_N/N$, and $x_\eig=x_{\eig,\rm l}+x_{\eig,\rm t} 1_N/N$ with $t_{\eig,\rm l}^\top1_N=s_{\eig,\rm l}^\top1_N=x_{\eig,\rm l}^\top1_N=0$ and $t_{\eig,\rm t},s_{\eig,\rm t},x_{\eig,\rm t}\in\R$. Due to the invariance, either under the conjugate action of $\U(N)$ on $\Herm(N)$ or the symmetric group $S_N$ on $\R^N$, it is immediate that
	 \begin{equation}
	 \begin{split}
	 \int_{\sqrt{\tr  X^2}\in(a,b)} X_{\rm t}\,X_{\rm l}\,F(\dv X)=&\frac{I_N}{N}\int_{\sqrt{\tr  X^2}\in(a,b)} X_{\rm t}\,\tr X_{\rm l}\,F(\dv X)=0,\\
	  \int_{\sqrt{\tr  X^2}\in(a,b)} x_{\eig,\rm t}\,x_{\eig,\rm l}\,f_{\eig}(\dv x_{\eig})=&\frac{1_N}{N}\int_{\|x_{\eig}\|\in(a,b)} x_{\eig,\rm t}\, x_{\eig,\rm l}^\top 1_N \,f_{\eig}(\dv x_{\eig})=0
	  \end{split}
	 \end{equation}
	 for any $0\leq a<b\leq\infty$. The measure $f_\eig$ is the induced measure for the eigenvalues of $X$. Thus, conditions~\eqref{4.2.2a} with $\Sigma$ and $F$ replaced by $\tilde{\sigma}^2I_N +\tilde{\kappa}1_N1_N^\top$ and $f_\eig$ and the second condition in~\eqref{4.2.1.b} read as follows
	 \begin{equation}\label{ave.split}
	 \begin{split}
	 \lim_{R\to0}\frac{\int_{\|x\|<R}(t_{\rm l}^\top x_{\eig,\rm l})^2f_\eig(\dv x_{\eig})+(t_{\rm t}/N)^2\int_{\|x\|<R}x_{\eig,\rm t}^2f_\eig(\dv x_{\eig})}{\int_{\|x\|<R}(s_{\rm l}^\top x_{\eig,\rm l})^2f_\eig(\dv x_{\eig})+(s_{\rm t}/N)^2\int_{\|x\|<R}x_{\eig,\rm t}^2f_\eig(\dv x_{\eig})}=&\frac{\tilde{\sigma}^2\|t_{\rm l}\|^2+(\tilde{\sigma}^2/N+\tilde{\kappa})t_{\rm t}^2}{\tilde{\sigma}^2\|s_{\rm l}\|^2+(\tilde{\sigma}^2/N+\tilde{\kappa})s_{\rm t}^2},\\
	 \lim_{R\to0}\frac{\int_{\sqrt{\tr X^2}<R}(\tr  T_{\rm l} X_{\rm l})^2F(\dv X)+(T_{\rm l}/N)^2\int_{\sqrt{\tr X^2}<R}X_{\rm t}^2F(\dv X)}{\int_{\sqrt{\tr X^2}<R}(\tr  S_{\rm l} X_{\rm l})^2F(\dv X)+(S_{\rm l}/N)^2\int_{\sqrt{\tr X^2}<R}X_{\rm t}^2F(\dv X)}=&\frac{\sigma^2\tr  T_{\rm l}^2+(\sigma^2/N+\kappa)T_{\rm t}^2}{\sigma^2\tr  S_{\rm l}^2+(\sigma^2/N+\kappa)S_{\rm t}^2}.
	  \end{split}
	 \end{equation}
	 We have not employed $s_\eig$ and $t_\eig$ here since $t=t_{\rm l}+t_{\rm t}1_N/N,s=s_{\rm l}+s_{\rm t}1_N/N\in\R^N$ with $t_{\rm l}^\top1_N=s_{\rm l}^\top1_N=0$ are not the eigenvalues of $T$ and $S$, respectively.
	 
	 Next, we perform an eigenvalue decompositions $X=U\diag(x_\eig)U^\dagger$, $S=V_1\diag(s_\eig)V_1^\dagger$ and $T=V_2\diag(t_\eig)V_2^\dagger$ with $x_\eig,s_\eig,t_\eig\in\R^N$ and $U,V_1,V_2\in\U(N)/\U^N(1)$ for the matrix integral. Since $U$ is Haar distributed it can absorb the unitary matrices $V_1$ and $V_2$, respectively. Employing~\cite[Theorem 2.1]{Benoit} with $q=d=2$ in their notation where the sum over $\sigma$, $\tau$ is over $S_2=\mathbb{Z}_2$ and $\lambda$ a sum over the two partitions $(0,\ldots,0,1,1)$ and $(0,\ldots,0,2)$, we can compute the unitary integral that is needed for the second equation in~\eqref{ave.split},
	 \begin{equation}
	 \bigl\langle(\tr  \diag(t_{\eig,\rm l}) U\diag(x_{\eig,\rm l})U^\dagger)^2\bigl\rangle=\frac{1}{N^2-1}\|t_{\eig,\rm l}\|^2\|x_{\eig,\rm l}\|^2
	 \end{equation}
	 with the notation~\eqref{am_functions}. One can also derive this formula by an explicit computation
	 \begin{equation}
	 \begin{split}
	 \bigl\langle(\tr  \diag(t_{\eig,\rm l}) U\diag(x_{\eig,\rm l})U^\dagger)^2\bigl\rangle=&\sum_{a,b,c,d=1}^Nt_{{\rm eig,l},a}t_{{\rm eig,l},c}x_{{\rm eig,l},b}x_{{\rm eig,l}, d}\langle |U_{ab}|^2|U_{cd}|^2\rangle\\
	 =&\|t_{{\rm eig,l}}\|^2\|x_{{\rm eig,l}}\|^2\bigl(\langle |U_{11}|^4\rangle-2 \langle |U_{11}|^2|U_{12}|^2\rangle+\langle |U_{11}|^2|U_{22}|^2\rangle\bigl),
	 \end{split}
	 \end{equation} 
	 where we used $t_{{\rm eig,l}}^\top1_N=x_{{\rm eig,l}}^\top1_N=0$ and the left and right $S_N$-invariance and the invariance under $U\to U^\top$ of the Haar measure of $\U(N)$. The three expectation values can be traced back to the moments of a single matrix because of
	 \begin{equation}
	 \begin{split}
	 1=&\left\langle\left(\sum_{j=1}^N|U_{1j}|^2\right)^2\right\rangle=N\left\langle|U_{11}|^4\right\rangle+N(N-1)\left\langle|U_{11}|^2|U_{12}|^2\right\rangle,\\
	\left\langle|U_{11}|^2\right\rangle=&\left\langle|U_{11}|^2\sum_{j=1}^N|U_{2j}|^2\right\rangle=\left\langle|U_{11}|^2|U_{12}|^2\right\rangle+(N-1)\left\langle|U_{11}|^2|U_{22}|^2\right\rangle.
	 \end{split}
	 \end{equation} 
	 The moment of a single matrix entry of $U\in\U(N)$ is the same as of a single component of a complex $N$-dimensional unit vector uniformly distributed on the sphere, meaning for $m\in\mathbb{N}_0$
	 \begin{equation}
	 \left\langle|U_{11}|^{2m}\right\rangle=\frac{m!(N-1)!}{(m+N-1)!}.
	 \end{equation}
	 
	 There is a similar equation when averaging over the symmetric group $S_N$ on the left hand side of the first equation in~\eqref{ave.split}, namely
	 \begin{equation}
	 \begin{split}
	 \frac{1}{N!}\sum_{\rho\in S_N}(t_{\rm l}^\top x_{\eig,\rm l,\rho})^2=&\sum_{a,b,c,d=1}^Nt_{{\rm l},a}t_{{\rm l},c}x_{{\rm eig,l},b}x_{{\rm eig,l}, d} \frac{1}{N!}\sum_{\rho\in S_N}\delta_{\rho(a)b}\delta_{\rho(c)d}\\
	 =&\|t_{{\rm l}}\|^2\|x_{{\rm eig,l}}\|^2\left(\frac{1}{N!}\sum_{\rho\in S_N}\delta_{\rho(1)1}+\frac{1}{N!}\sum_{\rho\in S_N}\delta_{\rho(1)1}\delta_{\rho(2)2}\right)\\
	 =&\frac{1}{N-1}\|t_{\rm l}\|^2\|x_{\eig,\rm l}\|^2.
	 \end{split}
	 \end{equation}
	 For this purpose, we exploit the invariance of $f_\eig$ under $S_N$. Thus the two equations read
	 \begin{equation}\label{ave.split.b}
	 \begin{split}
	 \lim_{R\to0}\frac{\frac{\|t_{\rm l}\|^2}{N-1}\int_{\|x_{\eig}\|<R}\|x_{\eig,\rm l}\|^2f_{\rm eig}(\dv x_{\eig})+\frac{t_{\rm t}^2}{N^2}\int_{\|x_{\eig}\|<R}x_{\eig,\rm t}^2f_\eig(\dv x_\eig)}{\frac{\|s_{\rm l}\|^2}{N-1}\int_{\|x_{\eig}\|<R}\|x_{\eig,\rm l}\|^2f_{\rm eig}(\dv x_{\eig})+\frac{s_{\rm t}^2}{N^2}\int_{\|x_{\eig}\|<R}x_{\eig,\rm t}^2f_\eig(\dv x_\eig)}=&\frac{\tilde{\sigma}^2\|t_{\rm l}\|^2+(\tilde{\sigma}^2/N+\tilde{\kappa})t_{\rm t}^2}{\tilde{\sigma}^2\|s_{\rm l}\|^2+(\tilde{\sigma}^2/N+\tilde{\kappa})s_{\rm t}^2},\\
	 \lim_{R\to0}\frac{\frac{\|t_{\eig,\rm l}\|^2}{N^2-1}\int_{\|x_{\eig}\|<R}\|x_{\eig,\rm l}\|^2f_{\rm eig}(\dv x_{\eig})+\frac{t_{\eig,\rm t}^2}{N^2}\int_{\|x_{\eig}\|<R}x_{\eig,\rm t}^2f_\eig(\dv x_\eig)}{\frac{\|s_{\eig,\rm l}\|^2}{N^2-1}\int_{\|x_{\eig}\|<R}\|x_{\eig,\rm l}\|^2f_{\rm eig}(\dv x_{\eig})+\frac{s_{\eig,\rm t}^2}{N^2}\int_{\|x_{\eig}\|<R}x_{\eig,\rm t}^2f_\eig(\dv x_\eig)}=&\frac{\sigma^2\|t_{\eig,\rm l}\|^2+(\sigma^2/N+\kappa)t_{\eig,\rm t}^2}{\sigma^2\|s_{\eig,\rm l}\|^2+(\sigma^2/N+\kappa)s_{\eig,\rm t}^2}.
	  \end{split}
	 \end{equation}
	 Comparison of the two equations with $t_{\rm l}=t_{\eig,\rm l}/\sqrt{N+1}$, $s_{\rm l}=s_{\eig,\rm l}/\sqrt{N+1}$, $t_{\rm t}=t_{\eig,\rm t}$, $s_{\rm t}=s_{\eig,\rm t}$, $\tilde{\sigma}^2=(N+1)\sigma^2$ and $\tilde{\kappa}=\kappa-\sigma^2$ yields the equivalence. We would like to point out that the uniqueness of the Haar measure of $\U(N)/\U^N(1)$ is also employed for this conclusion.
\end{proof}

\subsection{Heavy-tailed Random Matrices}\label{sec:heavy}

As we now have a feeling what the stable distributions and their domains of attraction are, we want to relate them to heavy-tailed random matrices. For this purpose, we define this kind of matrices, first.

\begin{definition}[Heavy-tailed Random Vectors and Matrices]\label{p3.2.2}\

	A random vector $x\in\R^d$ is called heavy-tailed if it has the finite index
\begin{equation}\label{3.2.1}
\tilde{\alpha}=\sup\left\{m\geq0\left|\E\left[\sup_{a\in\mathbb S^{d-1}}|a^\top x|^{m}\right]<\infty\right.\right\}<\infty.
\end{equation}

	A Hermitian random  matrix $X\in\Herm(N)$ is heavy-tailed when the index
\begin{equation}\label{3.2.1.b}
\tilde{\alpha}=\sup\left\{m\geq0\left|\E\left[\sup_{A\in\mathbb S_{\Herm(N)}}|\tr A X|^{m}\right]<\infty\right.\right\}<\infty.
\end{equation}
\end{definition}

These definitions are equivalent to some other statements, in particular those that involve other norms. Essentially, these equivalences only reflect that the norms generate the same topology.

\begin{lemma}[Equivalence to other norms]\label{p3.2.1}\

	Let $x\in\R^N$ be a random vector in $\R^d$, $m>0$, and $q\in[1,\infty)\cup\{\infty\}$. The following four statements are equivalent
	\begin{enumerate}
		\item $\displaystyle\E \left[\sup_{a\in\mathbb S^{d-1}}|a^\top x|^{m}\right]<\infty$;
		\item $\displaystyle\E \left[|a^\top x|^{m}\right]<\infty$ for any $a\in\R^d$;
		\item $\displaystyle\E [|x_j|^m]<\infty$ for all $j=1,\ldots,d$;
		\item $\displaystyle\E [\|x\|_q^m]<\infty$,
	\end{enumerate}
where $ \|a\|_q$ is the $l^q$-norm.
	In particular, the definition of $\tilde{\alpha}$ with each of the above norms is equivalent, i.e. $\tilde\alpha$ is the supremum of the set of the index $m$ such that any one of the previous four statements is satisfied.
\end{lemma}

\begin{proof}
	The implication (1)$\Rightarrow$(2) is trivial as for $a=0$ the average vanishes and for $a\neq0$ we divide $\E \left[|a^\top x|^{m}\right]$ by the $m$th power of the Euclidean norm of $a$, and exploit $|a^\top x|^{m}/\|a\|_2^m\leq\sup_{a\in\mathbb S^{d-1}}|a^\top x|^{m}$.
	
	The claim (2)$\Rightarrow$(3) follows from choosing the vector $a$ to have a $1$ at the $j$-th position and $0$ otherwise.
	
	 For (3)$\Rightarrow$(4), one notices that
	 \begin{equation}
	 \E [\|x\|_\infty^m]=\E\left[\left(\max_{j=1,\ldots,d}|x_j|\right)^m\right]=\sum_{j=1}^d\E [|x_j|^m\chi_{x_1,\ldots,x_{d-1}<x_d}]<\sum_{j=1}^d\E [|x_j|^m]<\infty
	 \end{equation}
	 with $\chi$ the indicator function.
	 Moreover, each $l_p$-norm generates the same topology meaning that when $\E [\|x\|_\infty^m]<\infty$ it holds $\E [\|x\|_q^m]<\infty$ for all $1\leq q\leq \infty$.
	 
	This equivalence of norms allows us to choose $q=2$ when proving  (4)$\Rightarrow$(1). We make use of the Cauchy-Schwarz inequality to obtain
	\begin{equation}
	\E\left[ \sup_{a\in\mathbb S^{d-1}}|a^\top x|^{m}\right]\le \E \left[\sup_{a\in\mathbb S^{d-1}}\|a\|_2^m\|x\|_2^m\right]=\E[\|x\|_2^{m}]<\infty.
	\end{equation}
	This finishes the proof.
\end{proof}

In the case of random matrices, Lemma~\eqref{p3.2.1} with $q=2$ actually states that
\begin{equation}\label{heavy-equi}
\begin{split}
\tilde{\alpha}=&\sup\left\{m\geq0\left|\E\left[\sup_{A\in\mathbb S_{\Herm(N)}}|\tr A X|^{m}\right]<\infty\right.\right\}\\
=&\sup\left\{m\geq0\left|\E\left[|\tr A X|^{m}\right]<\infty\ {\rm for\ all}\ A\in\Herm(N)\right.\right\}\\
=&\sup\left\{m\geq0\left| \E[|X_{ab}|^{m}]<\infty\ {\rm for\ all}\ a,b=1,\ldots,N\right.\right\}\\
=&\sup\left\{m\geq0\left| \E[(\tr X^2)^{m/2}<\infty\ {\rm for\ all}\ a,b=1,\ldots,N\right.\right\}.
\end{split}
\end{equation}
with $x_\eig\in\R^N$ the eigenvalues of $X$. Thence, any of these equations can be used to find the heavy-tail exponent $\tilde{\alpha}$.

We would like to underline that $\tilde{\alpha}$ should not be confused with the stability exponent $\alpha$. As we will show at the end of this subsection, being in the domain of attraction of a stable distribution with the stability exponent $\alpha$ implies $\alpha=\tilde{\alpha}$ but the converse is not true. Indeed, $\tilde{\alpha}$ is always defined as an element of $[0,\infty)$ while $\alpha$ is only an element in $(0,2]$ and is only defined when a corresponding central limit theorem exists.

Before we come to this point, we would like to highlight that the definition of $\tilde{\alpha}$ agrees for invariant random matrices with the one of the diagonal entries or the eigenvalues.

\begin{corollary}[Equivalence to Diagonal Entries and Eigenvalues]\label{heavy_tail_ensemble}\

	Let $X$ be an invariant random matrix with its diagonal entries $x_\diag\in\R^N$ and its eigenvalues $x_\eig\in\R^N$. Then the following statements are equivalent:
	\begin{enumerate}
		\item $X$ has a heavy tail with index $\tilde{\alpha}$;
		\item $x_\diag$ has a heavy tail with index $\tilde{\alpha}$;
		\item $x_{\rm eig}$ has a  heavy tail with index $\tilde{\alpha}$.
	\end{enumerate}
\end{corollary}

\begin{proof}
	The equivalence of (1) and (2), is clear due to the second line of~\eqref{heavy-equi}. Any $A\in\Herm(N)$ can be diagonalised and the invariance of the random matrix under the conjugate action of $\U(N)$ absorbs the eigenvectors of $A$. Thus, it holds
	\begin{equation}
	\E\left[|\tr A X|^{m}\right]= \E\left[|\tr a_\eig X|^{m}\right]=\E\left[|\,a_\eig^\top x_\diag|^{m}\right]
	\end{equation}
	with $a_\eig\in\R^N$ the eigenvalues of $A$ and $x_\diag$ the diagonal entries of $X$. This equation is exactly our claim when combined with Lemma~\ref{p3.2.1}.
	
	The equivalence of (1) and (3) is also a consequence of the invariance and the fourth line of~\eqref{heavy-equi}. Indeed the corresponding equation is now
	\begin{equation}
	\E\left[|\tr X^2|^{m/2}\right]=\E\left[\|x_\eig\|_2^{m}\right]
	\end{equation}
	in combination with the uniqueness of the Haar measure for $\U(N)/\U^N(1)$ which shows the equivalence.
\end{proof}

In the introduction, we have pointed out a particular class of invariant ensembles which have a rich analytical and algebraic structure called P\'olya ensembles. There are different kinds of these ensembles, multiplicative ones~\cite{KK2019,KK2016,FKK2021,KFI2019,Kieburg2020,KIeburg2019,KLZF2020} and additive ones~\cite{KR2016,FKK2021,Kieburg2019}.  We are naturally interested in those which are related to the additive convolution on $\Herm(N)$ that have a joint probability density of the eigenvalues~\cite[Definition 2.3]{FKK2021}
\begin{equation}
p(x)=\frac{1}{\prod_{j=0}^N j!}\Delta(x)\det[(-\partial_{x_a})^{b-1}\omega(x_a)]_{a,b=1,\ldots,N},
\end{equation}
in particular it holds $f_{\diag}(x)=\prod_{j=1}^N\omega(x_j)$ with $\omega$ a suitably differentiable and integrable probability density on $\R$ for the joint density of the diagonal entries in~\eqref{2.3.2}, see~\cite{KR2016,FKK2021,Kieburg2019}. Thus, a natural question is whether those additive P\'olya ensembles can be heavy-tailed. Unfortunately, this is not the case as the positivity of the probability density enforces an exponential tail decay as we show in the next statement.

\begin{proposition}[No-go Statement for P\'olya Ensembles]\label{Polya_no_heavy_tail}\

	For $N\ge 2$, P\'olya ensembles do not have a heavy tail.
\end{proposition}

This proposition implies that the Gaussian Unitary Ensemble (GUE) is the only stable P\'olya ensemble.

\begin{proof}
	Let us recall that for a P\'olya ensemble the marginal distribution of the diagonal entries factorises~\cite[Example 1(1)]{ZK2020},
	\begin{equation}
	f_\diag(x)=\prod_{j=1}^Nw(x_j)
	\end{equation}
	with $w$ a univariate probability density $N-1$ times differentiable. In~\cite{KR2016,KK2016}, it was shown that $w$ must be a P\'olya frequency function~\cite{Karlin,Polya,Polyab,Schoenberg} of order $N$, especially it has to satisfy~\cite[Definition 2.7]{FKK2021}
	\begin{equation}
	\Delta(x)\Delta(y)\det[w(x_a-y_b)]_{a,b=1,\ldots,n}\geq0
	\end{equation}
	for all $x,y\in\R^n$ and $n=1,\ldots,N$. A consequence is that it also holds~\cite{KK2019} and~\cite[Theorem 2.9]{FKK2021}
	\begin{equation}
	\Delta(x)\det[(-\partial_{x_a})^{b-1}w(x_a)]_{a,b=1,\ldots,n}\geq0
	\end{equation}
	for all $x\in\R^n$ and $n=1,\ldots,N$. Note that there is no ordering needed of the components $x_a$ as the Vandermonde determinant $\Delta(x)$ takes the sign. For $n=2$, we obtain that $w$ is log-concave, in particular it holds
	\begin{equation}
	\partial_{x_2}\log[w(x_2)]\geq \partial_{x_1}\log[w(x_1)]
	\end{equation}
	whenever $x_1\geq x_2$. Combined with the absolute integrability of $w$ on $\R$ implying $\log[w(x)]\to-\infty$ for $|x|\to\infty$, there are two constants $C,\tau>0$ such that $w(x)\leq Ce^{-\tau |x|}$ for all $x\in\R$. Therefore, all generalised moments exist.
\end{proof}

Let us turn to the relation between stable laws and heavy tail random variables. This connection was discussed in \cite[Pg. 179 Thm. 3]{GKC}. It has been shown that if a random variable is in the domain of attraction of some stable distribution with stability exponent $\alpha$, then it has a heavy tail with index $\tilde{\alpha}$ which is at least $\alpha$. Actually, it can be strengthened to a statement that indeed $\tilde{\alpha}=\alpha$. Despite we think that this fundamental statement has been shown somewhere in the literature, we could not find it. Therefore, we prove it here for the multivariate case.

\begin{proposition}[Relation between Domains of Attraction and Heavy-Tailed Random Vectors]\label{heavy_tail}\

	If an $\R^d$ random vector $x$ is in the domain of attraction of some stable distribution with stability exponent $0\le \alpha<2$, it has heavy tail with index $\tilde{\alpha}=\alpha$.
\end{proposition}

\begin{proof}
	The proof consists of two parts. The first part is to show that $\alpha\leq\tilde{\alpha}$. Due to Lemma~\ref{p3.2.1}, this inequality is equivalent with
	\begin{equation}
		\E[ |x_j|^\delta]<\infty,\quad{\rm for\ all}\ j=1,\ldots,d\quad{\rm and}\quad \delta<\alpha.
	\end{equation}
	which indeed holds because each marginal distributions for the entry $x_j$ of the random vector $x\in\R^d$ is in the domain of attraction of a stable distribution with the stability exponent $\alpha$ (as a simple corollary of that fact that marginal distributions of a stable distribution is again stable with the same stability parameter see also~\cite[Thm. 2.1.2]{ST94}). Then we can use the result~\cite[Pg. 179 Thm. 3]{GKC} which states that these generalised moments exist.
	
	In the second part of the proof we show that the generalised moments do not exist for any exponent $\delta>\alpha$. This can again be reduced to the univariate case because of Lemma~\ref{p3.2.1}. Let us assume that $X\in\R$ is a scalar random variable which satisfies 
	\begin{equation}\label{4.3.7}
		\lim_{R\rightarrow \infty}\frac{\mathbb P(|X|>kR)}{\mathbb P(|X|>R)}=k^{-\alpha}
	\end{equation}
	for any $k>0$. We combine this with the Markov inequality for
	\begin{equation}
		\E \left[|X|^\delta\right]\ge (kR)^\delta\,\mathbb P(|X|> kR),
	\end{equation}
	which holds true for any $k>0$ and $R>0$. Thus multiplying~\eqref{4.3.7} with $k^\delta$ and extending the ratio by $R^\delta$ and additionally choosing $k>1$ we have
	\begin{equation}
		\lim_{R\rightarrow \infty}\frac{(kR)^\delta\,\mathbb P(|X|>kR)}{R^\delta\,\mathbb P(|X|>R)}=k^{\delta-\alpha}>k^{(\delta-\alpha)/2}.
	\end{equation}
	As we have the limit we can find an $R_{k}>0$ such that for all $R>R_{k}$, it is
	\begin{equation}\label{4.3.11}
		\frac{(kR)^\delta\,\mathbb P(|X|>kR)}{R^\delta\,\mathbb P(|X|>R)}>k^{(\delta-\alpha)/2}.
	\end{equation}
	When choosing $k=2$ we see that $R>R_2$ also encompasses $2^jR>R_2$ for any $j\in\mathbb{N}_0$. Therefore,  substituting $R\to2^j R$ and $k=2$ in~\eqref{4.3.11} we have
	\begin{equation}
		\frac{(2^{j+1}R)^\delta\,\mathbb P(|X|>2^{j+1}R)}{(2^{j}R)^\delta\,\mathbb P(|X|>2^{j}R)}>2^{(\delta-\alpha)/2}.
	\end{equation}
	This can be rewritten in terms of a telescopic product when multiplying this inequality from $j=0$ up to $j=m-1$, yielding
	\begin{equation}
		\frac{(2^{m}R)^\delta\,\mathbb P(|X|>2^{m}R)}{R^\delta\,\mathbb P(|X|>R)}>2^{(\delta-\alpha)m/2}\quad\Leftrightarrow\quad (2^{m}R)^\delta\,\mathbb P(|X|>2^{m}R)>2^{(\delta-\alpha)m/2} R^\delta\,\mathbb P(|X|>R)
	\end{equation}
	for all $R>R_2$ and $m\in\mathbb{N}$. 
	This implies for the Markov-inequality
	\begin{equation}
		\E [|X|^\delta]>(2^{m}R)^\delta\,\mathbb P(|X|>2^{m}R)>2^{(\delta-\alpha)m/2} R^\delta\,\mathbb P(|X|>R)
	\end{equation}
	for any $R>R_2$ and $m\in\mathbb{N}$. Taking the limit $m\to\infty$ on the right hand side shows that the moment is unbounded because of $\delta>\alpha$. This implies that $\alpha\geq\tilde{\alpha}$. In combination with the first part of the proof, we have $\alpha=\tilde{\alpha}$ which has been our assertion.
\end{proof}

The following corollary is the specialisation of this proposition in terms of invariant random matrix ensembles on $\Herm(N)$. One can easily see this by making use of the equivalent definition given in Proposition~\ref{heavy_tail_ensemble} part (2).

\begin{corollary}[Relation between Domains of Attraction and Heavy-Tailed Random Matrices]\label{DoA_heavy_tail}\

	An invariant random matrix ensemble that is in the domain of attraction of some stable ensemble with stability exponent $0<\alpha<2$ has a heavy tail with index $\tilde{\alpha}=\alpha$. 
\end{corollary}

Finally let us point out, once again, that the condition~\eqref{4.1.2} for the domain of attraction is also referred to as regular variation~\cite{Basrak_thesis,Basrak}.  Thus, Proposition~\ref{heavy_tail} and Corollary~\ref{DoA_heavy_tail} tell us that regular variation implies a heavy tail with a very specific index. As not each measure is of regular variation it already shows that not each heavy-tailed matrix with an index $\tilde{\alpha}$ is in the domain of attraction of a stable ensemble with stability exponent $\alpha=\tilde{\alpha}$. To illustrate this we would like to consider a very simple univariate example. It is a probability measure of a sum of Dirac delta functions acting as follows on a bounded function $\varphi\in B(\R)$
	\begin{equation}
		\E[\varphi]=\sum_{j=1}^\infty 2^{-j}\varphi(2^j) \quad{\rm or\ equivalently}\quad \mathbb{P}(X=2^j)=2^{-j}\ {\rm with}\  j\in\mathbb{N}\ {\rm and\ vanishes\ otherwise}.
	\end{equation}
	It is clear that it has generalised moments up to order $1$. However it can be verified that this does not satisfy~\eqref{4.1.2}. For example, taking $\Sigma=\{1\}$ in~\eqref{4.1.2} it holds
	\begin{equation}
		\lim_{j\to\infty}\frac{\mathbb{P}(X> 2^{k+j})}{\mathbb{P}(X>2^j)}=\frac{1}{2^{\lceil k\rceil}}\neq \frac{1}{2^k}
	\end{equation}
	for any $k\geq 0$. We have exploited the floor function $\lceil k\rceil$ which yields the smallest integer which is larger than or equal to $k$. The reason of a break down of the statements above is that significant parts of probability are distributed on events that are too far apart. This prevents some kind of self-averaging effect that lies behind the generalised central limit theorems.
	
All distributions we have used so far are regularly varying when they have heavy tails.  For instance, let us consider the invariant elliptical distribution, see Proposition~\ref{prop:Gauss.ellip},
\begin{equation}\label{6.1.0}
P_p(X|y_0,\Sigma):=\int_0^\infty P(X|\sqrt{t},y_01_N,\Sigma)p(\dv t)
\end{equation}
with $\Sigma>0$ as in~\eqref{Sigma}, $y_0\in\R$, and $p$ being a Borel probability measure on $\mathbb{R}_+$. Let us denote by $\ast$ the additive convolution of two distributions, then it can be readily shown that
\begin{equation}
P_{p_1}(.|y_0,\Sigma)\ast P_{p_2}(.|y_1,\Sigma)=P_{p_1\ast p_2}(.|y_0+y_1,\Sigma).
\end{equation}
This means that the convolution of invariant elliptical random matrix ensembles can be traced back to the convolution of the corresponding univariate distribution. Therefore, when $p$ is in the domain of attraction of the univariate Levi $(\alpha/2)$-stable distribution $p_\alpha$, given by~\eqref{uniform.Levy}, then $P_p(.|y_0,\Sigma)$ is in the domain of attraction of the stable random matrix ensemble with distribution $P_{p_\alpha}(.|0,\Sigma)$. Note, that the origin of taking twice the stability exponent of $p_\alpha$ is due to the Gaussian form of the characteristic function of~\eqref{elliptic.Gaussian}, cf., Proposition~\ref{prop:Gauss.ellip}.

\section{Concluding Remarks}\label{s6}

We classified and analysed stable random matrix ensembles on $\Herm(N)$ that are invariant under the conjugate action of $\U(N)$ at fixed matrix dimension. In particular, we have traced back the suitable and necessary conditions to the corresponding marginal statistics for the eigenvalues as well as the diagonal entries. Both sets of quantities contain the full information about the spectral statistics since the eigenvectors are always distributed via the Haar measure on the coset $\U(N)/\U^N(1)$. This omission of redundant information can play a crucial role in further investigations, for instance the asymptotic analysis in the limit of large matrix dimensions $N\to\infty$. Although the highly involved expression of the characteristic function must be first overcome. Specific spectral quantities are macroscopic level density as well as the local spectral statistics like the level spacing distribution.

A first full classification of  the macroscopic level density that correspond to free variables has been already performed in~\cite{BPB99}. The point is that one needs almost surely asymptotically free stable random matrices, see~\cite[Chapter 4.1]{MS2017} for the definition, to be sure to find these kinds of level densities. For instance, some elliptical stable invariant random matrix ensembles do not belong to this class as has been discussed in~\cite[Section 4.2]{KM2021}. Then, one obtains deviations from the free probability results. Nevertheless, the macroscopic level density is still stable under matrix addition though it is not freely stable any more. Therefore, one needs to be careful to identify all stable random matrix ensembles we have classified in the present work with freely stable ones. The latter will be only a subset of all stable ones, and it is a task for further investigations to identify the necessary and sufficient conditions that the spectral measure has to satisfy for this subset. An intimately related question will be about the domain of attraction of this subset. Certainly this would be even more involved as double scaling limits between the number of matrices $m$ added and the matrix dimension $N$ are possible. It is quite likely that other emergent spectral statistics are possible that can be approached that are not covered when taking first $m\to\infty$ (as we have done in the present work) and, then, $N\to\infty$. Especially  the phenomena of an emerging non-trivial mesoscopic statistics observed in~\cite{KM2021} seems to corroborate this guess.

Another direction of study is to generalise the discussion in the present article to the other nine symmetry classes in the Altland-Zirnbauer classification scheme~\cite{AZ,Zirn}. Indeed it seems to be rather straightforward to three of those, namely random matrices invariant under their respective conjugate group actions that are imaginary anti-symmetric, Hermitian anti-self-dual, and Hermitian chiral. The reason is that there are derivative principle for all three classes~\cite{ZK2020}. Those can be traced back to known group integrals. In the present case of invariant Hermitian matrices this has been the Harish-Chandra-Itzykson-Zuber integral~\cite{Harish-Chandra,Itzykson-Zuber}. For the other classes the corresponding group integrals are only explicitly known for very particular cases so that it is very hard to reduce the random matrix statistics to the eigenvalues.

In~\cite{Gangolli} and subsequent works~\cite{Applebaum,Liao}, studies for infinitely divisible distributions on symmetric spaces have been performed. In some special cases, those distributions can be understood as being invariant under certain Lie group action. They have shown theorems similar to Theorem~\ref{p3.1}, that invariant infinitely divisible distributions also have invariant parameters. It is well-known that the stable distribution is a special kind of infinitely divisible distributions. Therefore, it is curious to see whether Theorem~\ref{p3.1} fit into their studies, and moreover, whether it is adaptable to the generalisations of our results to the Altland-Zirnbauer classification scheme.

\begin{acks}[Acknowledgments]
 We thank Zdzis\l aw Burda, Tianshu Cong and Holger K\"osters for fruitful discussion. We especially thank Holger K\"osters for proofreading the first version. We also thank Tianshu Cong for  ideas to prove Theorem~\ref{DoA}. Additionally, we are grateful for the very detailed feedback of the anonymous referee.
\end{acks}

\begin{funding}
JZ acknowledges financial support of a Melbourne postgraduate award, an ACEMS top up scholarship and FWO Flanders projects EOS 30889451. MK is supported by  the Australian Research Council via the grant DP210102887.
\end{funding}

\, 
\end{document}